\documentclass[11pt, dvipsnames]{article}
\usepackage{jmlrlocal}
\usepackage{header}
\usepackage[margin=1.0in]{geometry}
\title{Convergence Analysis of Gradient-Based Learning with Non-Uniform Learning Rates in Non-Cooperative Multi-Agent Settings}
\author{\name Benjamin Chasnov \email bchasnov@uw.edu\\
\addr Department of Electrical and Computer Engineering\\
    University of Washington\\
    \AND
    \name Lillian J. Ratliff \email ratliffl@uw.edu \\
    \addr Department of Electrical and Computer Engineering\\
       University of Washington \\    
       \AND
\name Eric Mazumdar \email mazumdar@eecs.berkeley.edu\\
\addr Department of Electrical Engineering and Computer Sciences\\
University of California, Berkeley\\
       \AND
       \name Samuel A. Burden \email sburden@uw.edu \\
         \addr Department of Electrical and Computer Engineering\\
       University of Washington}
\ShortHeadings{}{Chasnov, Ratliff, Mazumdar, and Burden}
\newcommand{\tg}{\tilde{g}}
\begin{document}

\maketitle

\begin{abstract}
    Considering a class of gradient-based multi-agent learning algorithms in
    non-cooperative settings, we provide local convergence guarantees to a
    neighborhood of a \emph{stable} local Nash equilibrium. In particular, we consider
    continuous games where agents learn in (i) deterministic settings with
    oracle access to their gradient and (ii) stochastic settings with an
    unbiased estimator of their gradient. Utilizing the minimum and maximum
    singular values of the \emph{game Jacobian}, we provide finite-time
    convergence guarantees in the deterministic case. On the other hand, in the
    stochastic case, we provide concentration bounds guaranteeing that with high
    probability agents will converge to a neighborhood of a stable local Nash
    equilibrium in finite time. Different than other works in this vein, we also
    study the effects of non-uniform learning rates on the learning dynamics and
    convergence rates. We find that much like preconditioning in optimization,
    non-uniform learning rates cause a distortion in the vector field which can,
    in turn, change the rate of convergence and the shape of the region of
    attraction. 
The analysis is supported by numerical examples that illustrate different
aspects of the theory. We conclude with discussion of the results and open
questions. 
\end{abstract}

\section{Introduction}
\label{sec:introduction}
The characterization and computation of equilibria such as \emph{Nash equilibria} and its refinements
constitutes a  significant focus in non-cooperative game theory. Several 
natural questions  arises including ``how do players find 
such equilibria?'' and ``how should the learning process be
interpreted?'' With these questions in mind, a variety of fields have focused
their attention on the problem of learning in games. This has, in turn, lead to a plethora
of learning algorithms including gradient play, fictitious play, best response,
and multi-agent reinforcement learning
among others~\cite{fudenberg:1998aa}.

From an applications point of view, a more recent trend is in the adoption of game theoretic models of algorithm interaction  in
machine learning applications.  
For instance, game theoretic
tools  are  being  used  to  improve  the  robustness  and
generalizability of machine learning algorithms; e.g., generative adversarial networks have become a popular topic
of  study  demanding  the  use  of  game  theoretic  ideas  to
provide  performance  guarantees 
\cite{daskalakis:2017aa}. 
In other work from the learning community, game theoretic concepts are being
leveraged to analyze the interaction of learning agents---see, e.g.,
\cite{heinrich:2016aa,mazumdar:2018aa,balduzzi:2018aa,tuyls:2018aa, mertikopoulos:2019aa}. 
Even more recently, convergence analysis to Nash equilibria has been called into
question~\cite{papadimitriou:2018aa}; in its place is a proposal to consider 
game dynamics as the \emph{meaning of the
game}. This is an interesting perspective as it is
well known that in general learning dynamics do not obtain an Nash equilibrium
even asymptotically---see, e.g.,~\cite{hart:2003aa}---and, perhaps more interestingly, many learning dynamics
exhibit very interesting limiting behaviors including periodic orbits and
chaos---see, e.g.,~\cite{benaim:1999ab,benaim:2012aa,hommes:2012aa,hofbauer:1996aa}.

Despite this activity, we still lack a complete understanding of the dynamics and limiting behaviors
of coupled, competing learning algorithms. 
One may imagine that the myriad results on convergence of gradient
descent in
optimization readily extend to the game setting. 
Yet, they do not since
gradient-based learning schemes in games \emph{do not correspond to gradient
flows}, a class of flows that are guaranteed to converge to local minimizers
almost surely. 
In
particular, the gradient-based learning dynamics for competitive, multi-agent
settings have a \emph{non-symmetric Jacobian} and as a consequence their dynamics may admit
complex eigenvalues and non-equilibrium limiting behavior such as periodic
orbits.
In short, this fact makes it difficult to extend many of the optimization
approaches to  convergence in single-agent optimization settings to multi-agent
settings primarily due to the fact that steps in the direction of individual
gradients of players' costs do not guarantee that each agents cost decreases.
In fact, in games, as our examples highlight, a player's cost can 
increase when they follow the gradient of their
own cost. Counterintuitively, agents can also converge to local maxima of their own costs despite descending their own gradient. These behaviors are due to the coupling between the agents.

Some of the questions that remain unaddressed and to which we provide
partial answers include the derivation of error bounds and convergence rates. These are important for 
ensuring performance guarantees on the collective behavior and can help provide guarantees on subsequent control or incentive
policy synthesis. We also investigate the question of how naturally arising
features of the learning process for autonomous agents, such as their learning
rates, impact the learning path and limiting behavior. This further exposes
interesting questions about the overall quality of the limiting behavior and the
cost accumulated along the learning path---e.g., is it better to be a slow or
fast learner both in terms of the cost of learning and the learned behavior? %
\paragraph{Contributions.} We study
convergence of a broad class of gradient-based multi-agent learning algorithms
in non-cooperative settings by leveraging the framework of $n$-player continuous
games along with tools from numerical optimization and dynamical systems theory. We consider a class of learning algorithms 
\[x_i^+=x_i-\gamma_i g_i(x_i,x_{-i})\]
where $x_i$ is the choice variable or action of player $i$, $\gamma_i$ is its learning rate, and $g_i$ is derived from the gradient of a function that abstractly
represents the cost of player $i$. The key feature of non-cooperative settings is coupling of an agent's cost through all other agents' choice variables $x_{-i}$.

We consider two settings: (i)
agents have oracle access to $g_i$ and (ii) agents have an unbiased estimator
for $g_i$. The class of gradient-based learning algorithms we study encompases a
wide variety of approaches to learning in games including multi-agent policy
gradient, gradient-based approaches to adversarial learning, and multi-agent
gradient-based online optimization.
For both the deterministic (oracle gradient access) and the stochastic (unbiased
estimators) settings, we provide convergence results for both uniform learning
rates---i.e., where $\gamma_i=\gamma$ for each player $i\in\{1, \ldots,
n\}$---and for non-uniform learning rates.  The latter of which arises more
naturally in the study of the limiting behavior of autonomous learning agents.

In the deterministic setting, we derive asymptotic and finite-time convergence rates for the
coupled learning processes to a refinement of local Nash equilibria known as
differential Nash equilibria~\cite{ratliff:2016aa} (a class of equilibria that are generic amongst
local Nash equilibria). In the stochastic setting, leveraging the results of
stochastic approximation and dynamical systems, we derive asymptotic convergence guarantees
to stable local Nash equilibria as well as high-probability, finite-time
guarantees for convergence to a neighborhood of a Nash equilibrium. 
The analytical results are supported by several illustrative numerical examples.
We also provide discussion on the effect of non-uniform learning rates on the
learning path---that is, different learning rates \emph{warp} the vector field
dynamics. Coordinate based learning rates are typically leveraged in gradient-based optimization
schemes to speed up convergence or avoid poor quality local minima. In games,
however, the interpretation is slightly different since each of the coordinates of the
dynamics corresponds to minimizing a different cost function along the
respective coordinate axis. The resultant effect is 
a distortion of the  vector field 
in such a way that it has the effect of leading the joint action to a point which has a lower
value for the \emph{slower player} relative to the flow of the dynamics given a
uniform learning rate and the same initialization. 
In this sense, it seems that the answer to the question posed above is that 
it is most beneficial for an agent to have the
slower learning rate.

\paragraph{Organization.} The remainder of the paper is organized as follows. We
start with mathematical and game-theoretic preliminaries in
Section~\ref{sec:prelims} which is followed by the main convergence results for
the deterministic setting (Section~\ref{sec:deterministic}) and the stochastic
setting (Section~\ref{subsec:stochastic_fintesample}). Within each of the latter
two sections, we present convergence results for both the case where agents have uniform and
non-uniform learning rates. In Section~\ref{sec:examples}, we present several
numerical examples which help to illustrate the theoretical results and also
highlight some directions for future inquiry. Finally, we conclude with
discussion and future work in Section~\ref{sec:discussion}.

\section{Preliminaries}
\label{sec:prelims}
Consider a setting in which at iteration $k$, each agent $i\in
\mc{I}=\{1,\ldots,n\}$ updates their choice variable $x_i\in X_i=\R^{d_i}$ by the process
\begin{equation}
    \begin{aligned}
        x_{i,k+1} &=  x_{i,k} - \gamma_{i,k} g_i(x_{i,k},x_{-i,k}).
    \end{aligned}
    \label{eq:simgrad}
\end{equation}
where $\gamma_i$ is agent $i$'s learning rate, $x_{-i}=(x_j)_{j\in \mc{I}/\{i\}}\in \prod_{j\in \mc{I}/\{i\}}X_j$ denotes the choices of all agents
excluding the $i$-th agent, and $(x_i,x_{-i})\in X=\prod_{i\in \mc{I}} X_i$.
Within the above setting, the class of learning algorithms we
consider is such that for each $i\in \mc{I}$, there exists a sufficiently smooth
function $f_i\in C^q(X,\mb{R})$, $q\geq 2$ such that $g_i$ is either $D_if_i$,
where $D_i(\cdot)$ denotes the derivative with respect to $x_i$, or an
unbiased estimator of $D_if_i$---i.e., $g_i\equiv \widehat{D_if_i}$ where
$\mb{E}[\widehat{D_if_i}]=D_if_i$.

The collection of costs $(f_1,\ldots, f_n)$ on $X=X_1\times \cdots \times X_n$ where
 $f_i:X\rar \mb{R}$ is agent $i$'s cost function and $X_i= \mb{R}^{d_i}$ is their action
 space defines a \emph{continuous game}. In this continuous game abstraction, each player $i\in
\mc{I}$ aims
to selection an action $x_i \in X_i$ that minimizes their cost $f_i(x_i,
x_{-i})$ given the actions of all other agents, $x_{-i}\in X_{-i}$.  That is, players myopically update their actions by
following the gradient of their cost with respect to their own choice variable.
For a symmetric matrix $A\in
\mb{R}^{d\times d}$, let
$\lambda_d(A)\leq \cdots \leq \lambda_1(A)$ be its eigenvalues. For a matrix $A\in \mb{R}^{d\times d}$, let $\spec(A)=\{\lambda_j(A)\}$ be the
spectrum of $A$.
  \begin{assumption}
      For each $i\in \mc{I}$, $f_i\in C^r(X, \mb{R})$ for $r\geq 2$ and $\omega(x)\equiv (D_1f_1(x)\ \cdots \ D_nf_n(x))$ is
     $L$--Lipschitz.
     \label{ass:fass}
 \end{assumption}

  Let $D_i^2f_i$ denote the second partial derivative of $f_i$ with respect to
 $x_i$ and $D_{ji}f_i$ denote the partial derivative of $D_if_i$ with respect to
 $x_j$.
 The \emph{game Jacobian}---i.e., the Jacobian of $\omega$---is given by
\[J(x)=\bmat{D_{1}^2f_1(x) & \cdots &D_{1n}f_1(x)\\\vdots & \ddots &
    \vdots\\
    D_{n1}f_n(x) &\cdots &
    D_{n}^2f_n(x)}.\]
The entries of the above matrix are dependent on $x$, however, we drop this
dependence where obvious.
Note that each $D_{i}^2f_i$ is symmetric under Assumption~\ref{ass:fass}, yet
$J$ is not. This is an important point and causes the subsequent analysis to
deviate from the typical analysis of  (stochastic) gradient
descent.

The most common characterization of limiting behavior in games is that of a Nash
equilibrium. The following definitions are useful for our analysis.

\begin{definition}
  \label{def:SLNE}
  A strategy $x\in X$ is a {local Nash equilibrium}
  for the game $(f_1, \ldots, f_n)$ if for each $i\in\mc{I}$ there exists
  an open set $W_i\subset X_i$ such
  that $x_{i}\in W_i$ and 
  $f_i(x_i,x_{-i})\leq f_i(x_i',x_{-i})$
  for all $x_{i}'\in W_i$.
    If the above inequalities are strict, 
  $x$ is a {strict local Nash equilibrium}. 
\end{definition}
\begin{definition}
A point $x\in X$ is said to be a {critical point} for the game if
    $\omega(x)=0$. 
\end{definition}
We denote the set of critical points as $\mc{C}=\{x\in X|\ \omega(x)=0\}$.
Analogous to single-player optimization settings, for each player, viewing all other
players' actions as fixed, there are necessary and sufficient
conditions which characterize local optimality.
  \begin{proposition}[\cite{ratliff:2016aa}] If $x$ is a local Nash equilibrium
      of the game $(f_1, \ldots, f_n)$, then 
$\omega(x)=0$ and $D_{i}^2f_i(x)\geq 0$.  On the other hand, if $\omega(x)=0$ and
$D_{i}^2f_i(x)>0$, then $x\in X$ is a local Nash equilibrium.
  \end{proposition}
The sufficient conditions in the above result give rise to the following definition of a
differential Nash equilibrium.
\begin{definition}[\cite{ratliff:2016aa}]
  \label{def:DNE}
  A strategy $x\in X$ is a {differential Nash equilibrium}
  if $\omega(x)=0$  and  $D^2_{i}f_i(x)>0$ for each $i\in\mc{I}$. 
\end{definition}
Differential Nash need not be isolated. However, if
 $J(x)$ is
non-degenerate---meaning that $\det J(x)\neq 0$---for a differential Nash $x$,
then $x$ is an \emph{isolated strict local Nash equilibrium}. 
Non-degenerate differential Nash are \emph{generic} amongst local Nash
equilibria and they are \emph{structurally
stable}~\cite{ratliff:2014aa} which ensures they persist under small perturbations.
This result also
implies an asymptotic convergence result:
 if the spectrum of $J$  is strictly in
the right-half plane (i.e.~$\spec(J(x))\subset\mb{C}_+^\circ$), then a differential Nash equilibrium $x$ is
(exponentially) attracting under the flow of
$-\omega$~\cite[Proposition~2]{ratliff:2016aa}. We say such equilibria are
\emph{stable}.

\section{Deterministic Setting}
\label{sec:deterministic}
The multi-agent  learning framework we analyze is such that each
agent's rule for updating their choice variable consists of the agent modifying
their action $x_i$ in the direction of their individual gradient $D_if_i$.
 Let us first consider the setting in which  each agent $i$ has oracle access to
$g_i$. 
The learning dynamics are given by
\begin{equation}
    x_{k+1}=x_k-\Gamma \omega(x_k)\label{eq:det}
\end{equation}
where $\Gamma=\mathrm{blockdiag}(\gamma_1I_{d_1}, \ldots, \gamma_nI_{d_n})$ with 
$I_{d_i}$ denoting the $d_i\times d_i$ identity matrix.
Within this setting we consider both the cases where the agents have a
constant \emph{uniform} learning rate---i.e., $\gamma_i\equiv \gamma$---and where
their learning rates are \emph{non-uniform}, but constant---i.e., $\gamma_i$ is
not necessarily equal to $\gamma_j$ for any $i,j\in\mc{I}$, $j\neq i$.

 Let 
 $S(x)=\frac{1}{2}(J(x)+J(x)^T)$ be the symmetric part of
 $J(x)$.
 Define \[\alpha=\min_{x\in B_r(x^\ast)}\lambda_{d}\big(
 S(x)^TS(x) \big)\]
and \[\beta=\max_{x\in
B_r(x^\ast)}\lambda_{1}(J(x)^TJ(x))\] where $B_r(x^\ast)$ is a $r$--radius
    ball around $x^\ast$. 
     For a stable differential Nash $x^\ast$, let $B_r(x^\ast)$ be a ball of
radius $r>0$
around the equilibrium $x^\ast$ that is contained in the region of attraction
$\mc{V}(x^\ast)$ 
for $x^\ast$\footnote{Many techniques exists for approximating the region of
attraction; e.g., given a Lyapunov function, its
largest invariant level set can be used
as an approximation~\cite{sastry:1999aa}.  Since
$\mathrm{spec}(J(x^\ast))\subset \mb{C}_\circ^+$, the converse Lyapunov theorem guarantees the existence of a
local Lyapunov function.}.
Let $B_{r_0}(x^\ast)$ with $0<r_0<\infty$ be the \emph{largest ball} contained in the
region of attraction of $x^\ast$.
\subsection{Uniform Learning Rates}
\label{subsec:uniform_deterministic}
With $\gamma_i=\gamma$ for each $i\in \mc{I}$, 
the  learning rule \eqref{eq:det} can be thought of as a discretized numerical scheme
 approximating the continuous time dynamics \[\dot{x}=-\omega(x).\]
With a judicious choice of learning rate $\gamma$,
 \eqref{eq:det} will converge (at an exponential rate) to a locally stable equilibrium of the dynamics.

\begin{proposition}
    Consider an $n$--player continuous game $(f_1, \ldots, f_n)$ satisfying
    Assumption~\ref{ass:fass}. 
    Let $x^\ast\in X$ be a stable differential Nash equilibrium. Suppose agents
    use the gradient-based learning rule $x_{k+1}=x_k-\gamma \omega(x_k)$
    with learning rates  $0<\gamma<\tilde{\gamma}$ where $\tilde{\gamma}$ is the
  smallest positive $h$
    such that
    $\max_j|1-h\lambda_j(J(x^\ast))|=1$.
     Then, 
     for $x_0\in B_r(x^\ast)\subset \mc{V}(x^\ast)$, $x_k\rar x^\ast$
     exponentially.
    \label{thm:convergenceNonoise}
\end{proposition}

The above result provides a range for the possible learning rates for which
\eqref{eq:det} converges to a stable differential Nash equilibrium $x^\ast$ of $(f_1,
\ldots, f_n)$ assuming agents initialize in a ball contained in the region of
attraction of $x^\ast$. 
Note that the usual assumption in gradient-based approaches to single-objective
optimization problems (in which case $J$ is symmetric) is that $\gamma<1/L$, where objective being minimized is $L$-Lipschitz. This is
sufficient to guarantee convergence since
the spectral radius of a matrix is always less than any operator norm which, in
turn,
ensures that $|1-\gamma\lambda_j|<1$ for each $\lambda_j\in
\spec(J(x^\ast))$. 
 If the game is a potential game---i.e.,
 there exists a function $\phi$ such that $D_if_i=D_i\phi$ for each $i$ which
 occurs if and only if $D_{ij}f_i=D_{ji}f_j$---then convergence analysis
 coincides with gradient descent so that any $\gamma<1/L$ where $L$ is the
 Lipschitz constant of $\omega$ results in local asymptotic convergence.

The convergence guarantee in Proposition~\ref{thm:convergenceNonoise} is
asymptotic in nature. It is often useful, from both an analysis and synthesis
perspective, to have non-asymptotic or finite-time convergence results. Such
results can be used to provide guarantees on decision-making processes wrapped
around the coupled learning processes of the otherwise autonomous agents. The next result,
provides a finite-time convergence guarantee for gradient-based learning where
 agents uniformly use a fixed step size. 

Let $B_r(x^\ast)$ be defined as before with the added condition that it be
 defined to be the largest ball in the region of attraction such that on $B_r(x^\ast)$ the symmetric part of
$J$---i.e.,
$S\equiv\frac{1}{2}(J+J^T)$---is positive definite. 
\begin{theorem}
    Consider a game $(f_1, \ldots, f_n)$ on $X=X_1\times\cdots\times X_n$
    satisfying Assumption~\ref{ass:fass}.
    Let $x^\ast\in X$ be a stable differential Nash
    equilibrium. Suppose $x_0\in B_r(x^\ast)$ and that $\alpha<\beta$. Then,
    given $\vep>0$, the gradient-based learning
    dynamics with learning rate 
    $\gamma= \sqrt{\alpha}/\beta$ 
    obtains an $\varepsilon$--differential Nash such
    that $x_k\in B_\vep(x^\ast)\subset B_r(x^\ast)$ for all \[k \geq \left\lceil
    2\frac{\beta}{\alpha}\log \frac{r}{\varepsilon}\right\rceil. \]
    \label{thm:2player}
\end{theorem}
Before we proceed to the proof, let us remark on the assumption that
$\alpha<\beta$. First, $\alpha\leq \beta$ is always true; indeed, suppressing the
dependence on $x$,
\begin{align*}
   \lambda_{d}( S^TS )
    \textstyle\leq
    \lambda_{1}(S^TS)
    &\textstyle
\leq\sigma_{\max}(J)^2=\lambda_{1}(J^TJ)
\end{align*}
where $\sigma_{\max}(\cdot)$ denotes the largest singular value of its argument.
Thus, the  condition that $\alpha<\beta$ 
is
generally true; 
for equality to
hold, the symmetric part of $J(x)$ would have \emph{repeated} eigenvalues, 
which is not generic.
Hence, we include this assumption in Theorem~\ref{thm:2player}, but note that it is not restrictive and is fairly benign.

\begin{proof}[Proof of Theorem~\ref{thm:2player}]
  First, note that $\|x_{k+1}-x^\ast\|=\|\tg(x_k)-\tg(x^\ast)\|$ where
        $\tg(x)=x-\gamma
    \omega(x)$. 
    Now, given $x_0\in B_r(x^\ast)$, by the mean value theorem,
    \[\textstyle\|\tg(x_0)-\tg(x^\ast)\|=\|\int_0^1 D\tg(\tau
    x_0+(1-\tau)x^\ast)(x_0-x^\ast)d\tau\|\leq \sup_{x\in
    B_r(x^\ast)}\|D\tg(x)\|\|x_0-x^\ast\|.\]
    Hence, it suffices to show that for the choice of $\gamma$, the eigenvalues of
$I-\gamma J(x)$
are in the unit circle.
Indeed, since $\omega(x^\ast)=0$, we have that
\begin{align*}
    \textstyle\|x_{k+1}-x^\ast\|_2&=\|x_k-x^\ast-\gamma(\omega(x_k)-\omega(x^\ast))\|_2  \textstyle\leq
    \sup_{x\in B_r(x^\ast)}\|I-\gamma J(x)\|_2\|x_k-x^\ast\|_2\end{align*}
If 
$\sup_{x\in B_r(x^\ast)}\|I-\gamma J(x)\|_2$ is less than one, then the
dynamics are contracting.
For notational convenience, we drop the explicit dependence on $x$.
Since $\lambda_d(S)\geq \sqrt{\alpha}$ on $B_r(x^\ast)$, \begin{align*}
     ( I -\gamma \Jac)^T &(I -\gamma J) 
     \leq (1-2\gamma \lambda_{d}(S)+\gamma^2\lambda_{1}(\Jac^T\Jac))I
     \leq\textstyle( 1-\frac{\alpha}{\beta})I
 \end{align*}
where the last inequality holds for  $\gamma=\sqrt{\alpha}/{\beta}$.
    Hence,
\begin{align*}
  \textstyle  \|x_{k+1}-x^\ast\|_2&  \textstyle \leq \sup_{x\in B_r(x^\ast)}\|I-\gamma
\Jac(x)\|_2\|x_k-x^\ast\|_2\leq  \textstyle 
    (1-\frac{\alpha}{\beta})^{1/2}\|x_k-x^\ast\|_2.
\end{align*}
Since $\alpha<\beta$, we have that $(1-\alpha/\beta)<
\exp(-\alpha/\beta)$ so that
    \[\|x_{T}-x^\ast\|_2\leq
    \exp(-T\alpha/(2\beta))\|x_0-x^\ast\|_2.\]
This, in turn, implies that $x_k\in B_\vep(x^\ast)$ for all $k\geq T=
\lceil2\frac{\beta}{\alpha} \log(r/\varepsilon)\rceil$.
\end{proof}

Note that $\gamma=\sqrt{\alpha}/{\beta}$ is selected to minimize $1-2\gamma
\lambda_{1}(S)+\gamma^2\lambda_{d}(\Jac^T\Jac)$.
Hence, this is the fastest learning rate given the worst case 
eigenstructure of $\Jac$ over
the ball $B_r(x^\ast)$ for the choice of operator norm $\|\cdot\|_2$. We note, 
however, that faster convergence is possible as indicated by
Proposition~\ref{thm:convergenceNonoise} and observed in the examples in
Section~\ref{sec:examples}. Indeed, we note that the spectral
radius $\rho(\cdot)$ of a matrix is always less than its maximum singular
value---i.e.~$\rho(I-\gamma\Jac)\leq \|I-\gamma \Jac\|_2$---so it
is possible to contract at a faster rate.
We remark that if $\Jac$ was symmetric (i.e.,~in the case
of a potential game~\cite{monderer:1996aa} or a single-agent optimization problem), then
$\rho(I-\gamma \Jac)=\|I-\gamma \Jac\|_2$.
In games, however,
$\Jac$ is not symmetric.

\subsection{Non-Uniform Learning Rates}
Let us now consider the case when agents have their own individual learning rate
$\gamma_i$, yet still have oracle access to their individual gradients. This is, of course, more natural in the study of autonomous learning
agents as opposed to efforts for computing Nash equilibria for a given game. 

\label{subsec:deterministic_asymptotic}
 \begin{proposition}
    Consider an $n$--player game $(f_1, \ldots, f_n)$ satisfying
    Assumption~\ref{ass:fass}. 
    Let $x^\ast\in X$ be a stable differential Nash equilibrium. Suppose agents
    use the gradient-based learning rule $x_{k+1}=x_k-\Gamma \omega(x_k)$
    with learning rates  $\gamma_i$ such that $\rho(I-\Gamma \Jac(x))<1$
    for all $x\in \mc{V}(x^\ast)$. 
     Then, 
     for $x_0\in \mc{V}(x^\ast)$, $x_k\rar x^\ast$ exponentially.
    \label{thm:convergenceNonoise_nonuniform}
\end{proposition}
The proof is a direct application of Ostrowski's
theorem~\cite{ostrowski:1966aa}. We provide a simple proof via Lyapunov argument
for posterity.

\citet{mazumdar:2018aa} show that
\eqref{eq:det} will almost surely avoid strict saddle points of the dynamics,
some of which are Nash equilibria in non-zero sum games. Note that the set of
critical points $\mc{C}$ contains
more than just the local Nash equilibria. Hence, except on a set of measure
zero, \eqref{eq:det} will converge to a stable attractor of $\dot{x}=-\omega(x)$
which includes stable limit cycles and stable local non-Nash critical points.

\label{subsec:deterministic_finitesample}

Letting $\tg(x)=x-\Gamma \omega(x)$, since $\omega\in C^q$ for some $q\geq 1$,
$\tg\in C^q$, the expansion
\[\tg(x)=\tg(x^\ast)+(I-\Gamma \Jac(x))(x-x^\ast)+R(x-x^\ast)\]
holds,
where $R$ satisfies
    $\lim_{x\rar x^\ast}\|R(x-x^\ast)\|/\|x-x^\ast\|=0$
so  that given $c>0$, there exists an $r>0$ such that
$\|R(x-x^\ast)\|\leq c\|x-x^\ast\|$ for all $x\in B_{r}(x^\ast)$.
\begin{proposition}
    Suppose that $\|I-\Gamma \Jac(x)\|<1$ for all $x\in B_{r_0}(x^\ast)\subset
    \mc{V}(x^\ast)$ so that there exists $r',r''$ such that $\|I-\Gamma
    \Jac(x)\|\leq r'<r''<1$ for all $x\in B_{r_0}(x^\ast)$. For $ 1-r''>0$, let $0<r<\infty$ be the
    largest $r$ such that
    $\|R(x-x^\ast)\|\leq (1-r'')\|x-x^\ast\|$ for all $x\in B_r(x^\ast)$.
    Furthermore, let $x_0\in B_{r^\ast}(x^\ast)$, where $r^\ast=\min\{r,r_0\}$, be
    arbitrary. Then, given $\vep>0$, gradient-based learning with learning rates
    $\Gamma$ obtains an $\vep$--differential Nash equilibrium in finite
    time---i.e., $x_k\in
    B_\vep(x^\ast)$ for all $k\geq T=\lceil \frac{1}{\delta}\log\left(
    r^\ast/\vep
    \right)\rceil$ where $\delta = r''-r'$.
    \label{prop:nonuniformone}
\end{proposition}
The proof follows the proof of Theorem 1 in \cite{argyros:1999aa} with a few
minor modifications; we provide it
in Appendix~\ref{app:proofs_deterministic} for completeness. 
\begin{remark}
    We note that the proposition can be more generally stated with the
    assumption that $\rho(I-\Gamma \Jac(x))<1$, in which case there exists
    some $\delta$ defined in terms of bounds on powers of $I-\Gamma \Jac$. We
    provide the proof of this in Appendix~\ref{app:proofs_deterministic}. We also note that
    these results hold even if $\Gamma$ is not a diagonal matrix as we have
    assumed as long as
    $\rho(I-\Gamma \Jac(x))<1$. 
\end{remark}

A perhaps more interpretable finite bound stated in terms of the game structure can also be
obtained. Consider the case in which players adopt
learning rates $\gamma_i=\sqrt{\alpha}/(\beta k_i)$ with $k_i\geq 1$. Given a
stable differential Nash equilibrium $x^\ast$, let
$B_r(x^\ast)$ be the largest ball of radius $r$ contained in the region of attraction on
which $\til{S}\equiv\frac{1}{2}(\til{\Jac}^T+\til{\Jac})$ is positive definite
where
$\til{\omega}=(D_if_i/k_i)_{i\in \mc{I}}$ so that $\til{\Jac}\equiv D\til{\omega}$, and define
\begin{align*}
    \til{\alpha}=\textstyle\min_{x\in B_r(x^\ast)}\lambda_{d}\big(
    \til{S}(x)^T\til{S}(x) \big)
\end{align*}
and 
\begin{equation*}
    \textstyle  \til{\beta}=\max_{x\in
  B_r(x^\ast)}\lambda_{1}(\til{J}(x)^T\til{J}(x)).
\end{equation*}
Given a stable differential Nash equilibrium $x^\ast$, let $B_r(x^\ast)$ be the
largest ball contained in the region of attraction $\mc{V}(x^\ast)$ on which
$S^TS$ is positive definite---i.e., 
$\sqrt{\alpha}>0$.
\begin{theorem}
  Suppose that Assumption~\ref{ass:fass} holds and that
    $x^\ast\in X$ is a stable differential Nash
    equilibrium. Let $x_0\in B_r(x^\ast)$, $\alpha<k_{\min}\beta$,
    $\sqrt{\alpha}/k_{\min}\leq \sqrt{\til{\alpha}}$, 
    and for each $i$, $\gamma_i=\sqrt{\alpha}/(\beta k_i)$ with $k_i\geq 1$.
Then,
    given $\vep>0$, 
    the gradient-based learning
    dynamics with learning rates 
    $\gamma_i$ 
    obtain an $\varepsilon$--differential Nash such
    that $x_k\in B_\vep(x^\ast)$ for all 
    \[k\geq   \left\lceil
    2\frac{\beta k_{\min}}{{\alpha}}\log\left(\frac{r}{\vep}\right)\right\rceil.\]
    \label{thm:multiplerates}
\end{theorem}
\begin{proof}
        First, note that $\|x_{k+1}-x^\ast\|=\|\tg(x_k)-\tg(x^\ast)\|$ where
        $\tg(x)=x-\Gamma
    \omega(x)$. 
    Now, given $x_0\in B_r(x^\ast)$, by the mean value theorem,
    \[\textstyle\|\tg(x_0)-\tg(x^\ast)\|=\|\int_0^1 D\tg(\tau
    x_0+(1-\tau)x^\ast)(x_0-x^\ast)d\tau\|\leq \sup_{x\in
    B_r(x^\ast)}\|D\tg(x)\|\|x_0-x^\ast\|.
\]
Hence, it suffices to show that for the choice of $\Gamma$, the eigenvalues of
$I-\Gamma J(x)$
live in the unit circle. Then an inductive argument can be made with the
inductive hypothesis that $x_k\in B_r(x^\ast)$. 
Let $\Lambda=\diag\left( 1/k_1, \ldots, 1/k_n \right)$. Then we need to show
that $I-\gamma \Lambda J$ has eigenvalues in the unit circle. 
Since $\omega(x^\ast)=0$, we have that
\[\|x_{k+1}-x^\ast\|_2=\|x_k-x^\ast-\gamma\Lambda(\omega(x_k)-\omega(x^\ast))\|_2
\leq \textstyle\sup_{x\in B_r(x^\ast)}\|I-\gamma \Lambda J(x)\|_2\|x_k-x^\ast\|_2.\]
If 
$\sup_{x\in B_r(x^\ast)}\|I-\gamma \Lambda J(x)\|_2$ is less than one,
where the norm is the operator $2$--norm, then the
dynamics are contracting.
For notational convenience, we drop the explicit dependence on $x$.
Then,
    \begin{align}
        (I-\gamma \Lambda J)^T(I-\gamma \Lambda J)\leq \textstyle(1-2\gamma
        \lambda_d(\til{S})+\frac{\gamma^2\lambda_1(J^TJ)}{k_{\min}^2})I\textstyle&\leq
        (1-2\gamma\sqrt{\alpha}/k_{\min}+\alpha/(\beta
            k_{\min}))I\notag\\
            &\textstyle =(1-\alpha/(\beta
            k_{\min}))I\notag.
    \end{align}
The first inequality holds since
    $\lambda_1(J^TJ/k_{\min}^2)\geq
    \lambda_1(J^T\Lambda^2J)$. Indeed, first observe that the
    singular values of $\Lambda J^TJ \Lambda$ are the same as those
    of $J^T\Lambda^2J$ since the latter is positive definite
    symmetric. Thus, by noting that $\|A\|_2=\sigma_{\max}(A)$ and employing
    Cauchy-Schwartz, we get that
    $\|\Lambda\|_2^2\|J^TJ\|_2\geq \|\Lambda
    J^TJ\Lambda\|_2$ and hence, the inequality. 
  Using the above to bound $\sup_{x\in B_r(x^\ast)}\|I-\gamma\Lambda
    J(x)\|_2$, we have
  $\|x_{k+1}-x^\ast\|_2\leq  \textstyle 
    (1-\frac{\alpha}{\beta k_{\min}})^{1/2}\|x_k-x^\ast\|_2$.
    Since $\alpha<k_{\min}\beta$, $(1-\alpha/(\beta k_{\min}))<
    e^{-\alpha/(\beta k_{\min})}$ so that
    $\|x_{k+1}-x^\ast\|_2\leq
    e^{-T\alpha/(2k_{\min}\beta)}\|x_0-x^\ast\|_2$.
This, in turn, implies that
 $x_k\in B_\vep(x^\ast)$ for all $k\geq T= \lceil2\frac{\beta k_{\min}}{\alpha} \log(r/\varepsilon)\rceil$.
\end{proof}

Multiple learning rates lead to a scaling rows which
can have a significant effect on the eigenstructure of the matrix,
thereby making the relationship between $\Gamma
J$ and $J$ difficult to reason about. None-the-less, there are numerous approaches to solving
nonlinear systems of equations (or differential equations expressed as a set of
nonlinear system of equations) that employ \emph{preconditioning} (i.e., coordinate scaling). The purpose of using a preconditioning matrix is to
rescale the problem and achieve better or faster convergence. Many of
these results directly translate to convergence guarantees for learning in games
when the learning rates are not uniform; however, in the case of understanding
convergence properties for autonomous agents learning an equilibrium---as
opposed to computing an equilibrium---the \emph{preconditioner} is not subject
to design. Perhaps this reveals an interesting direction of future research in
terms of synthesizing games or learning rules via incentivization or otherwise
exogenous control policies for either coordinating agents or improving the learning
process---e.g., using incentives to induce a particular equilibrium while
also encouraging faster learning.  
\section{Stochastic Setting}
\label{subsec:stochastic_fintesample}
In this section, we consider gradient-based learning rules for each agent where
the agent does not have oracle access to their individual gradients, but rather
has an unbiased estimator in its place. In particular, for each player
$i\in\mc{I}$,  consider the noisy gradient-based learning rule given by
\begin{equation}x_{i,k+1}=x_{i,k}-\gamma_{i,k}(\omega(x_k)+w_{i,k+1})\label{eq:stoch}\end{equation}
where $\gamma_{i,k}$ is the learning rate and $w_{i,k}$ is an independent identically
distributed stochastic process. In order to prove a high-probability, finite
sample convergence rate, we can leverage recent results for convergence of
nonlinear stochastic approximation algorithms. The key is in formulating the the
learning rule for the agents and in leveraging the notion of a stable differential Nash
equilibrium which has analogous properties as a locally stable equilibrium for
a nonlinear dynamical system. Making the link between the discrete time learning
update and the limiting continuous time differential equation and its equilibria
allows us to draw on rich existing convergence analysis tools.

In the first part of this section, we provide convergence rate results for the case where the agents use a
\emph{uniform learning rate}---i.e.~$\gamma_{i,k}\equiv \gamma_{k}$. In the
second part of this section, we extend these results to the case where agents
use \emph{non-uniform learning rates}---that is, each agent has its own learning 
rate $\gamma_{i,k}$---by incorporating some additional assumptions and
leveraging two-timescale analysis techniques from dynamical systems theory.

We require some modified assumptions in this section on the learning process
structure. 
\begin{assumption}
The gradient-based learning rule \eqref{eq:stoch} satisfies the following:
    \begin{enumerate}[itemsep=-2pt, topsep=0pt,
            label=\textbf{A2\alph*.}, leftmargin=35pt]
        \item  Given the filtration $\mc{F}_k=\sigma(x_s,w_{1,s}, w_{2,s}, s\leq k)$,
    $\{w_{i,k+1}\}_{i\in \mc{I}}$ are conditionally independent. Moreovoer, for
    each $i\in\mc{I}$,
    $\mb{E}[w_{i,k+1}|\ \mc{F}_k]=0$ almost surely (a.s.), and $\mb{E}[\|w_{i,k+1}\||\
    \mc{F}_{k}]\leq c_i(1+\|x_{k}\|)$ a.s.~for some constants
    $c_i\geq 0$.
\item   For each $i\in \mc{I}$, the stepsize sequence $\{\gamma_{i,k}\}_k$
    contain positive scalars
    such that
    \begin{enumerate}[topsep=0pt]
        \item $\sum_{i}\sum_{k}\gamma_{i,k}^2<\infty$;
        \item $\sum_k
    \gamma_{i,k}=\infty$;
\item and, $\gamma_{2,k}=o(\gamma_{1,k})$.   
\end{enumerate}
\item  Each $f_i\in C^q(\mb{R}^d,\mb{R})$ for some $q\geq 3$ and
    each $f_i$ and $\omega$ are $L_i$-- and $L_\omega$--Lipschitz,
    respectively.
       \end{enumerate}
    \label{ass:lip}
\end{assumption}

\subsection{Uniform Learning Rates}
\label{app:uniform}
Before concluding, we specialize to the case in which agents have
the same learning rate sequence $\gamma_{i,k}= \gamma_k$ for each $i\in
\mc{I}$.
\begin{theorem}
    Suppose that $x^\ast$ is a stable differential Nash equilibrium of the game
    $(f_1,\ldots, f_\np)$
    and that Assumption~\ref{ass:lip} holds (excluding A2b.iii). For each $\tind$,
    let $\no\geq 0$
    and 
    \[\textstyle\zeta_\tind=\max_{\no\leq s\leq \tind-1}\big( \exp(-\lambda
    \sum_{\ell=s+1}^{\tind-1}\gamma_\ell \big)\gamma_s.\]
    Fix any $\vep>0$ such that
    $B_{\vep}(x^\ast)\subset B_r(x^\ast)\subset \mc{V}$ where $\mc{V}$ is the
    region of attraction of $x^\ast$. There exists constants
$C_1, C_2>0$ and functions $h_1(\vep)=O(\log(1/\vep))$ and $h_2(\vep)=O(1/\vep)$
so that whenever $T\geq h_1(\vep)$ and $\no\geq N$, where $N$ is such that
$1/\gamma_\tind\geq h_2(\vep)$ for all $\tind\geq N$, the samples generated by the gradient-based learning rule
satisfy 
\begin{align*}
    &\textstyle\Pr\left( \bar{x}(t)\in B_\vep(x^\ast) \ \forall t\geq t_{\no}+T+1|\
    \bar{x}(t_{\no})\in B_r(x^\ast) \right)\\
    &\qquad\textstyle\geq 1-\sum_{s=\no}^\infty\big(
    C_1\exp(-C_2\vep^{1/2}/\gamma_s^{1/2})\textstyle+C_1\exp(-C_2\min\{\vep,
    \vep^2\}/\zeta_s)\big)
\end{align*}
where the constants depend only on parameters $\lambda, r, \tau_L$ and the
dimension $d=\sum_{i}d_i$. 
Then stochastic gradient-based learning in games obtains an
$\vep$--stable differential Nash $x^\ast$ in finite time with high probability.
    \label{thm:new}
\end{theorem}
The above theorem implies that $x_k\in B_\vep(x^\ast)$ for all $k\geq \no+\lceil
\log(4\tilde{K}/\vep)\lambda^{-1}\rceil+1$ with high probability for some constant $\tilde{K}$ that
depends only on $\lambda, r, \tau_L$, and $d$. 

\begin{proof}
    Since $x^\ast$ is a stable differential Nash equilibrium, $J(x^\ast)$
    is positive definite and $D_i^2f_i(x^\ast)$ is positive definite for each
    $i\in \mc{I}$. Thus $x^\ast$ is a locally asymptotically stable hyperbolic equilibrium
    point of $\dot{x}=-\omega(x)$. Hence, the assumptions of Theorem
    1.1~\cite{thoppe:2018aa} are satisfied so that we can invoke the result
    which gives us the high probability bound for stochastic gradient-based
    learning in games. 
\end{proof}

The above theorem has a direct corollary specializing to the case where the
gradient-based learning rule with uniform stepsizes is initialized inside a ball
of radius $r$ constained in the region of attraction---i.e.,
$B_r(x^\ast)\subset \mc{V}$. 
\begin{corollary}
    Let $x^\ast$ be a stable differential Nash equilibrium of $(f_1,\ldots,
    f_n)$ and suppose that Assumption~\ref{ass:lip} holds (excluding A2b.iii).  Fix any $\vep>0$ such that
    $B_{\vep}(x^\ast)\subset B_r(x^\ast)\subset \mc{V}$. Let $\zeta_\tind$, $T$,
    and $h_2(\vep)$ be defined as in Theorem~\ref{thm:new}. Suppose that
$1/\gamma_\tind\geq h_2(\vep)$ for all $\tind\geq 0$ and that $x_0\in
B_r(x^\ast)$. Then, with $C_1, C_2>0$ as in Theorem~\ref{thm:new},
\begin{align*}
    &\textstyle\Pr\left( \bar{x}(t)\in B_\vep(x^\ast) \ \forall t\geq T+1|\
    \bar{x}(t_{\no})\in B_r(x^\ast) \right)\\
    &\qquad\textstyle\geq 1-\sum_{s=0}^\infty\big(
    C_1\exp(-C_2\vep^{1/2}/\gamma_s^{1/2})\textstyle+C_1\exp(-C_2\min\{\vep,
    \vep^2\}/\zeta_s)\big).
\end{align*}
\end{corollary}
 
\subsection{Non-Uniform Learning Rates}
\label{subsec:stochastic_multiplelearn}

Consider now that agents have their own
learning rates $\gamma_{i,k}$ for each $i\in \mc{I}$. 
In environments with several
autonomous agents, as compared to the objective of \emph{computing} Nash equilibria in
a game, it is perhaps more reasonable to consider the scenario in
which the agents have their own individual learning rate. For the sake of brevity,
we show the convergence result in detail for the two agent case---that is, where
$\mc{I}=\{1,2\}$. We note that
the extension to $n$ agents is straightforward.
The proof leverages recent 
results from the theory of stochastic approximation presented in~\cite{borkar:2018aa} 
and we note that our objective here is to show that they apply to
games and provide commentary on the interpretation of the results in
this context.

The gradient-based learning rules  are given by
\begin{equation}
x_{i,k+1}=x_{i,k}-\gamma_{i,k}(\omega(x_k)+w_{i,k+1})
    \label{eq:noisyupdate}
\end{equation}
so that with $\gamma_{2,k}=o(\gamma_{1,k})$, in the limit $\tau \rar 0$, the above system can be thought of
as approximating the singularly perturbed system 
\begin{align}
    \bmat{\dot{x}_{1}(t)\\
    \dot{x}_2(t)}=
    -\bmat{ D_1f_1(x_1(t),x_2(t))\\
    \tau D_2f_2(x_1(t),x_2(t))}
    \label{eq:singperturb}
\end{align}
 Indeed, since $\lim_{k\rar\infty}\gamma_{2,k}/\gamma_{1,k}\rar
0$---i.e., $\gamma_{2,k}\rar 0$ at a faster rate than
$\gamma_{1,k}$---updates to $x_1$ appear to be equilibriated for
the current quasi-static $x_2$ as the dynamics in \eqref{eq:singperturb}
suggest.

\subsubsection{Asymptotic Convergence in the Non-Uniform Learning Rate Setting}
\begin{assumption}
    For fixed $x_2\in X_2$, the system $\dot{x}_1(t)=-D_1f_1(x_1(t),x_2)$ has a globally
    asymptotically stable equilibrium $\lambda(x_2)$.
    \label{ass:twotime-1}
\end{assumption}
\begin{lemma}
    Under Assumptions~\ref{ass:lip} and \ref{ass:twotime-1}, conditioned on the event $\{\sup_k\sum_i
        \|x_{i,k}\|_2<\infty\}$, $(x_{1,k}, x_{2,k})\rar \{(\lambda(x_2), x_2)|\
            x_2\in \mb{R}^{d_2}\}$ almost surely.
    \label{lem:asconverge}
\end{lemma}
The above lemma follows from classical analysis (see,
e.g.,~\citet[Chapter~6]{bokar:2008aa} or \citet[Chapter~3]{bhatnagar:2013aa}). 

Define the continuous time accumulated after
$k$ samples of  $x_2$ to be $\textstyle t_k=\sum_{l=0}^{k-1}\gamma_{2,k}$
and define $x_2(t,s,x_s)$ for $t\geq s$ to
be the trajectory of $\dot{x}_2=-D_2f_2(\lambda(x_2),x_2)$. Furthermore, define the event
$\mc{E}=\{\sup_k\sum_i
        \|x_{i,k}\|_2<\infty\}$.
\begin{theorem}
    Suppose that Assumptions~\ref{ass:lip} and \ref{ass:twotime-1} hold.
     For any $K>0$,  conditioned on $\mc{E}$,
        \[\textstyle \lim_{k\rar\infty} \sup_{0\leq h\leq K}\|x_{2,k+h}-x_2(t_{k+h}, t_k,
        x_{k})\|_2=0.\]
    \label{thm:converge}
\end{theorem}
\begin{proof}
    The proof invokes Lemma~\ref{lem:asconverge} above and Proposition~4.1 and
    4.2 of \cite{benaim:1999aa}. Indeed, by Lemma~\ref{lem:asconverge},
    $(\lambda(x_{2,k})-x_{2,k})\rar 0$ almost surely. Hence, we can study the sample path generated by 
    \[x_{2,k+1}=x_{2,k}-\gamma_{2,k}(D_2f_2(\lambda(x_{2,k}),x_{2,k})+w_{2,k+1}).\]
    Since $D_2f_2 \in C^{q-1}$ for some $q\geq 3$, it is locally
    Lipschitz and, on the event $\{\sup_k\sum_i
        \|x_{i,k}\|_2<\infty\}$, it is bounded.  It thus induces a continuous globally integrable
vector field, and therefore satisfies the assumptions of Proposition~4.1
of~\cite{benaim:1999aa}. Moreover, under Assumption~\ref{ass:lip}, the
assumptions of Proposition~4.2 of~\cite{benaim:1999aa} are satisfied. Hence,
invoking said propositions, we get the desired result.
\end{proof}
This result essentially says that the slow player's sample path asymptotically
tracks the flow of \[\dot{x}_2=-D_2f_2(\lambda(x_2),x_2).\]
If we additionally assume that the slow component also has a global attractor,
then the above theorem gives rise to a stronger convergence result. 
\begin{assumption}
    Given $\lambda(\cdot)$ as in Assumption~\ref{ass:twotime-1}, the system $\dot{x}_2(t)=-\tau D_2f_2(\lambda(x_2(t)),x_2(t))$ has  a globally
    asymptotically stable equilibrium $x_2^\ast$. 
    \label{ass:twotime-2}
\end{assumption}
\begin{corollary}
   Under the assumptions of Theorem~\ref{thm:converge} and
    Assumption~\ref{ass:twotime-2}, conditioned
   on the event $\mc{E}$, gradient-based learning converges almost surely to
   a stable attractor $(x_1^\ast,x_2^\ast)$, where $x_1^\ast=\lambda(x^\ast_2)$,
   the set of which contains the stable differential Nash equilibria.
\end{corollary}
More generally, the process $(x_{1,k},x_{2,k})$ will converge almost surely to the 
internally chain transitive set of the limiting dynamics \eqref{eq:singperturb}
and this set contains the stable Nash equilibria.
If the only internally chain transitive sets for \eqref{eq:singperturb} are
isolated equilibria (this occurs, e.g., if the game is a potential game), then
$x_k$ converges almost surely to a stationary point of the dynamics, a subset of which are
stable local Nash equilibria.

It is also worth commenting on what types of games will satisfy these assumptions.
To satisfy
Assumption~\ref{ass:twotime-1}, it is sufficient for the fastest player's cost
function to be convex
in their choice
variable. 

\begin{proposition}
    Suppose Assumption~\ref{ass:lip} and \ref{ass:twotime-2} hold and that     $f_1(\cdot,x_2)$ is convex.  
    Conditioned on the event $\mc{E}$, the sample points of gradient-based
        learning  satisfy $(x_{1,k}, x_{2,k})\rar \{(\lambda(x_2), x_2)|\
            x_2\in \mb{R}^{d_2}\}$ almost surely. Moreover, 
$(x_{1,k}, x_{2,k}) \rar
   (x_1^\ast,x_2^\ast)$ almost surely, where $x_1^\ast=\lambda(x^\ast_2)$.
   \label{cor:almostconvex}
\end{proposition}
Note that $(x_1^\ast,x_2^\ast)$ could still be a spurious stable non-Nash point
still since the above implies that
$D(D_2f_2(\lambda(\cdot),\cdot))|_{x_2^\ast}>0$,
which does not necessarily  imply that $D_2^2f_2(\lambda(x_2^\ast),x_2^\ast)>0$.

\begin{remark}[Relaxation to Local Asymptotic Stability.]
    Under relaxed assumptions on {global asymptotic stability},
we can obtain high-probability results on convergence to locally asymptotically
stable attractors. If it is assumed that $x_0$ is in the region of attraction
for a locally asymptotically stable attractor, then the above results can be
stated with only the assumption of a locally asymptotic stability. However, this
is difficult to ensure in practice. To relax the result to a local guarantee
regardless of the initialization requires conditioning on an unverifiable
event---i.e., the high-probability bound in this case is conditioned on the event
$\{ \{x_{1,k}\}$ belongs to a compact set $B$, which depends on the sample
point,  of $\cap_{x_2} \mc{R}(\lambda(x_2))\}$ where $\mc{R}(\lambda(x_2))$ is
the region of attraction of $\lambda(x_2)$. None-the-less, it is possible to
leverage results from stochastic
approximation~\cite{karmakar:2018aa}, \cite[Chapter~2]{bokar:2008aa} to prove local
versions of the results for non-uniform learning rates. Further investigation is
required to provide concentration bounds for not only games but stochastic
approximation in general.
\end{remark} 

\subsubsection{High-Probability, Finite-Sample Guarantees with Non-Uniform
Learning Rates}
In the stochastic setting, the learning dynamics are stochastic approximation
updates, and non-uniform learning rates lead to a multi-timescale setting. The
results  leverage recent theoretical guarantees for
two-timescale analysis of stochastic approximation such as~\cite{borkar:2018aa}.

For a stable differential Nash equilibrium
$x^\ast=(\lambda(x_2^\ast),x_2^\ast)$, using the bounds in Lemma~\ref{lem:defK}
and Lemma~\ref{lem:defbarK} in
Appendix~\ref{app:proofs_stocahstic}, we can provide a high-probability
guarantee that $(x_{1,k},x_{2,k})$
gets locked in to a ball around $(\lambda(x_2^\ast),x_2^\ast)$. 

Let $\bar{x}_i(\cdot)$ denote the linear interpolates between sample points
$x_{i,k}$ and, as in the preceding sub-section, let $x_i(\cdot,t_{i,k},x_k)$ denote
the continuous time flow of $\dot{x}_i$ with initial data $(t_{i,k},x_k)$ where
$t_{i,k}=\sum_{l=0}^{k-1}\gamma_{i,k}$. 
Alekseev's formula is a nonlinear
variation of constants formula that provides solutions to perturbations of
differential equations using a local linear approximation. We can apply it to
the \emph{asymptotic pseudo-trajectories} $\bar{x}_i(\cdot)$ in each timescale. For these local
approximations, linear systems theory
lets us find growth rate bounds for the perturbations, which can, in turn, be
used to bound the normed difference between the continuous time flow and the
asymptotic  pseudo-trajectories. More detail is provided in
Appendix~\ref{app:proofs_stocahstic}.

Towards this end, fix $\vep\in[0,1)$ and let
    $N$ be such that $\gamma_{1,\tind}\leq
    \vep/(8K)$ and $\tau_\tind\leq \vep/(8K)$ for all $\tind\geq N$. Define
    time
    sequences  
    $t_{1,k}=\tilde{t}_{k}$ and $t_{2,k}=\hat{t}_{k}$ which keep track of the
    time accumulated up to iteration $k$ on each of the timescales. Let $\no\geq
    N$ and, with
$K$ as in Lemma~\ref{lem:defK} (Appendix~\ref{app:proofs_stocahstic}), let $T$ be such that
\[e^{-\kappa_1(\tilde{t}_{\tind}-\tilde{t}_{\no})}H_{\no}\leq \vep/(8K)\] for all
$\tind\geq
\no+T$ where $\kappa_1>0$ is a constant derived from Alekseev's formula applied
to $\bar{x}_1(\cdot)$. Analogously, with $\bar{K}$ as in Lemma~\ref{lem:defbarK}
(Appendix~\ref{app:proofs_stocahstic}),  let
\[e^{-\kappa_2(\hat{t}_\tind-\hat{t}_{\no})}(\|\bar{x}_2(\hat{t}_{\no})-{x}_2(\hat{t}_{\no})\|\leq
\vep/(8\bar{K}),\] for all $\tind\geq \no+T$ where $\kappa_2>0$ is a constant derived from Alekseev's formula applied
to $\bar{x}_2(\cdot)$. Define constants 
\[\beta_\tind=\textstyle\max_{\no\leq s\leq
    \tind-1}\exp(-\kappa_1(\sum_{i=s+1}^{\tind-1}\gamma_{1,i}))\gamma_{1,s}, \ \textstyle\eta_\tind=\max_{\no\leq s\leq
    \tind-1}\big(\exp(-\kappa_2(\sum_{i=s+1}^{\tind-1}\gamma_{2,i}))\gamma_{2,s}\big),\]
 and $\tau_k=\gamma_{2,k}/\gamma_{1,k}$.
\begin{theorem}
    Suppose that Assumptions~\ref{ass:lip}--\ref{ass:twotime-2}
    hold and let $\gamma_{2,k}=o(\gamma_{1,k})$. %
      Given a stable
    differential Nash equilibrium $x^\ast=(\lambda(x_2^\ast),x_2^\ast)$, 
player 2's sample path (generated by \eqref{eq:noisyupdate} with $i=1$) will
    asymptotically track  $z_k=\lambda(x_{2,k})$. Moreover, given
    $\vep\in[0,1)$, $x_k$ will get `locked in' to a $\vep$--neighborhood with
    high probability conditioned on reaching  $B_{r_0}(x^\ast)$
    by iteration $\no$.
    That is, letting $\bar{\tind}= \no+T+1$, for some $C_1, C_2>0$,
    \begin{align}
        \mathrm{P}(\|x_{1,\tind}-z_\tind\|\leq \vep, \forall \tind\geq
        \bar{\tind}| x_{1,\no},z_{\no}\in
        B_{r_0})\geq&\textstyle
        1-\sum_{\tind=\no}^{\infty}C_1\exp\big(-C_2\sqrt{\vep}/\sqrt{\gamma_{1,\tind}}\big)\notag\\
        &\textstyle-\sum_{\tind=\no}^\infty
        C_2\exp\big(-C_2\sqrt{\vep}/\sqrt{\tau_\tind}\big)\notag\\
&\textstyle\quad -\sum_{\tind=\no}^{\infty} C_1\exp\big(-C_2\vep^2/\beta_\tind\big).
        \label{eq:pxz}
    \end{align}
     Moreover, for some constants $\tilde{C}_1, \tilde{C}_2>0$, 
    \begin{align}
        \mathrm{P}(\|x_{2,\tind}-x_{2}(\hat{t}_\tind)\|\leq \vep, \forall
        \tind
        \geq\bar{\tind}
        |x_{\no},z_{\no}\in B_{r_0}(x^\ast)) \textstyle \geq&\textstyle 1+\sum_{\tind=\no}^\infty
        \tilde{C}_1\exp\big(-\tilde{C}_2\sqrt{\vep}/\sqrt{\gamma_{1,\tind}}\big)\notag\\
        &\textstyle -\sum_{\tind=\no}^\infty
        \tilde{C}_1\exp\big(-\tilde{C}_2\sqrt{\vep}/\sqrt{\tau_\tind}\big)\notag\\
        &\quad\textstyle - \sum_{\tind=\no}^\infty
        \tilde{C}_1\exp\big(-\tilde{C}_2\vep^2/\beta_\tind\big)\notag\\
        &\qquad\textstyle - \sum_{\tind=\no}^\infty
        \tilde{C}_1\exp\big(-\tilde{C}_2\vep^2/\eta_\tind\big).
        \label{eq:py}
    \end{align}
    \label{thm:conjecturetrack}
\end{theorem}
\begin{corollary}
    Fix $\vep\in[0,1)$ and suppose that $\gamma_{1,n}\leq \vep/(8K)$ for all
        $n\geq 0$. With $K$ as in Lemma~\ref{lem:defK}
        (Appendix~\ref{app:proofs_stocahstic}), let $T$ be such that
$e^{-\kappa_1(\tilde{t}_n-\tilde{t}_{0})}H_{0}\leq \vep/(8K)$ for all $n\geq T$. 
Furthermore,  with $\bar{K}$ as in Lemma~\ref{lem:defbarK}
(Appendix~\ref{app:proofs_stocahstic}),  let
$e^{-\kappa_2(\hat{t}_n-\hat{t}_{0})}(\|\bar{x}_2(\hat{t}_{0})-{x}_2(\hat{t}_{0})\|\leq
\vep/(8\bar{K})$, $\forall n\geq T$. Under the assumptions of
Theorem~\ref{thm:conjecturetrack}, $x_{k}$ will will get `locked in' to a $\vep$--neighborhood with
high probability conditioned on $x_0\in B_{r_0}(x^\ast)$ where the
high-probability bounds in  \eqref{eq:pxz} holds with $\no=0$.
\label{cor:zerolock}
\end{corollary}
\begin{remark}[Relaxation to Locally Asymptotically Stable Attractors.]
    In fact, Corollary~\ref{cor:zerolock} holds under a relaxed assumption
    on the stability of $x^\ast$. Indeed, if $x^\ast$ is locally asymptotically
    stable and $x_0\in B_{r_0}(x^\ast)\subset \mc{R}(x^\ast)$ where
    $\mc{R}(x^\ast)$ is the
    region of attraction for $x^\ast$, then the high probability bound from
    Corollary~\ref{cor:zerolock} holds.
\end{remark}
The key technique in proving the above theorem---the complete details are
provided in~\citet{borkar:2018aa} which, in turn, leverages results from
\citet{thoppe:2018aa}---is first to compute the errors between the sample
points from the stochastic learning rules and the continuous time flow generated
by initializing the continuous time limiting dynamics at each sample point and
flowing it forward for time $t_{\tind+1}-t_\tind$, doing this for each $x_{1,k}$ and
$x_{2,k}$ separately and in their own timescale, and then take a union bound over all the continuous time
intervals defined for  $\tind\geq \no$.

\section{Numerical Examples}
\label{sec:examples}
The results in the preceding sections provide convergence guarnatees for a class
of gradient-based learning algorithms to a neighborhood of a {stable} Nash
equilibrium under deterministic and stochastic gradient-based  update rules with
both uniform and non-uniform learning rates. 
In this section, we present several numerical examples that validate these
theoretical results and highlight
interesting aspects of learning in multi-agent settings.

  \subsection{Deterministic  Policy Gradient in Linear Quadratic Dynamic Games}
The first example we explore is a linear quadratic (LQ) game with three players
in the space of linear feedback policies. This game serves as a useful benchmark
since it has a unique global equilibrium that we can compute via a set of
coupled algebraic Riccati equations~\cite{basar:1998aa}. The gradient-based
learning rule for each of the agents is a multi-agent version of policy
gradient in which agents have oracle access to their gradients at each
iteration.
\begin{figure}[t]
    \centering
    \newlength{\imagewidth}
    \settowidth{\imagewidth}{\includegraphics{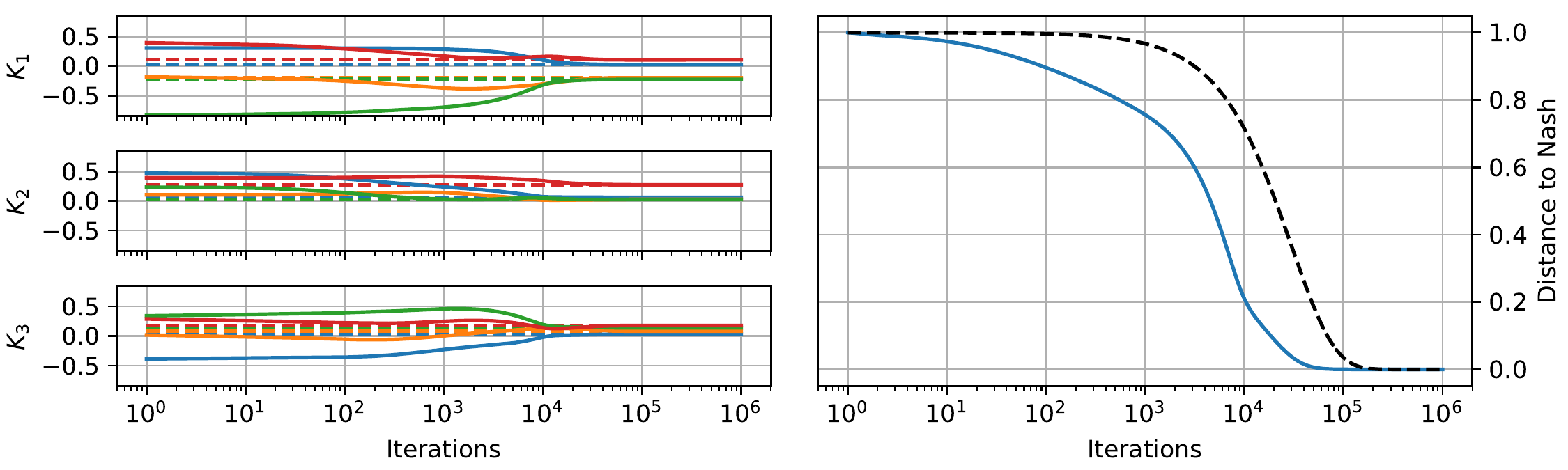}}
    \subfloat[][]{
        \includegraphics[trim = 0 0 0.5\imagewidth{} 0, clip=true, width=0.49\columnwidth]{figs/lqr_benchmark.pdf}
    }
    \subfloat[][]{
        \includegraphics[trim = 0.5\imagewidth{} 0 0 0, clip=true, width=0.49\columnwidth]{figs/lqr_benchmark.pdf}
        \label{fig:lqr_pg_b}
    } 
    \caption{Convergence of policy gradient in LQ dynamic games to the Nash policy. 
    (a) Each player's linear feedback gain matrix $K_i$ converges to the unique Nash policies (dotted lines). 
    (b) The black dashed line shows upper bound of the number of iterations required to converge within $\varepsilon$ distance from Nash (2-norm). The actual convergence for this random initialization is shown as the solid line. }
    \label{fig:lqr_pg}
\end{figure}

Consider a four state discrete time linear dynamical system,
\[
   z(t+1) = Az(t) + B_1 u_1(t) + B_2 u_2(t) + B_3 u_3(t)
\]
where $z(t) \in \mathbb R^4$ and, for each $i\in\{1,2,3\}$, $u_i(t) \in \mathbb
R$ is the control for player $i$. The policy for each player is parameterized by
a linear feedback gain matrix, $u_i(t) = -K_i z(t)$. Moreover, each player seeks 
to minimize a quadratic cost
\[
    f_i(K_i, K_{-i}) = \mb{E}_{z_0\sim \mc{D}}\left[\sum_{t=0}^\infty \left(
        z(t)^TQ_iz(t) + \textstyle\sum_{j=1}^n u_j(t)^T R_{ij} u_j(t)
    \right)\right]
\]
which is a function of the coupled state variable $z(t)$, their own control $u_i(t)$ and all other agents' control $u_{-i}(t)$ over an infinite time horizon.
In an effort to learn a Nash equilibrium, each agent employs policy gradient. In
particular, they update their feedback policy via
\[K_i(t+1)=K_i(t)-\gamma_i\nabla_{K_i}f_i(K_i, K_{-i}).\]
It is fairly straightforward to compute the gradient of $f_i$ with respect to
$K_i$, the feedback gain that parameterizes player $i$'s control input $u_i$.
Indeed, 
\[\nabla_{K_i}f_i(K_i,K_{-i})=2(R_{ii}K_i - B_i^T P_i \widetilde A
    )\Sigma_K\]
    where
    \[\Sigma_K=\mb{E}_{z_0\sim \mc{D}}\left[\sum_{t=0}^{\infty} z(t)
    z(t)^T\right].\]
Hence, the collection of the agents' individual gradients is given by 
\[\omega(K_1, K_2, K_3) = \left(
        2(R_{ii}K_i - B_i^T P_i \widetilde A )\Sigma_K
        \right)_{i=1}^3
\]
\begin{remark}
    Note that $\omega$ can be zero at critical points or at points where $\sum_{t=0}^{\infty} z(t) z(t)^T$ drops rank. To prevent the latter possibility, we
        sample the initial condition from a distribution. That is, we take
        $z_0\sim \mc{D}$ so that $\mb{E}_{z_0\sim \mc{D}}z_0z_0^T$ is full rank. 
\end{remark}
For a given joint policy $(K_1, K_2, K_3)$, the closed loop dynamics are $\widetilde
A = A-B_1K_1-B_2K_2-B_3K_3$. The states $z(t)$ are obtained from simulating the
system.
For each $i$, the Riccati matrix $P_i$ is computed by solving the Riccati equation
\[
    P_i = \widetilde A^T P_i \widetilde A + Q_i + \sum_{j=1}^n K_j R_{ij} K_j.
 \]  
 Note that this Riccati equation is only used to compute the gradient of the cost functions with respect to a specific set of feedback gains. 
The system parameters used in this example are listed in Appendix~\ref{app:lqr}.

For the purpose of validating convergence, we can compute the Nash policies
$(K_1^\ast, K_2^\ast, K_3^\ast)$ by an established method
with coupled Riccati equations, explained in Appendix~\ref{app:lqr}. 
We use the learning rate $\gamma_i=\gamma$ defined as in
Theorem~\ref{thm:2player}. To compute $\gamma$ we first compute the game
Jacobian $J(K_1^*,\ K_2^*,\ K_3^*)$ at the Nash feedback gains and then find the
maximum eigenvalue  of $J^T J$ and minimum eigenvalue of $(J^T + J)^T (J^T +J)$
in a neighborhood of $(K_1^\ast, K_2^\ast, K_3^\ast)$ to determine the constants
$\alpha$ and $\beta$ as defined in Section~\ref{sec:deterministic}.

Figure~\ref{fig:lqr_pg} shows the convergence of the gradient updates to the
Nash policies. The $K_i$ are randomly initialized in a neighborhood of the known
Nash equilibrium and such that $\tilde{A}$ is stable.
 The number of iterations required to converge to an $\varepsilon$--differential
 Nash is bounded by the dashed black line in Figure~\ref{fig:lqr_pg_b}, which
 shows the curve of $(\varepsilon, T)$ pairs determined by
 Theorem~\ref{thm:2player}. However, this learning rate is not optimal, as
 choosing a larger $\gamma$ will result in faster convergence as empirically
 observed. 

 \begin{remark}[Stochastic Policy Gradient] We note that stochastic policy
     gradient with an unbiased estimator has similar convergence properties.
     Here, e.g., the state dynamics may be subject to zero-mean, finite-variance
     noise. As long as the estimator for the gradient is unbiased, the
     theoretical guarantees of the proceeding sections apply.
 \end{remark}

\subsection{Benchmark: matching pennies}
\begin{figure}[t!]
  \centering \centering
  \includegraphics[width=.8\columnwidth]{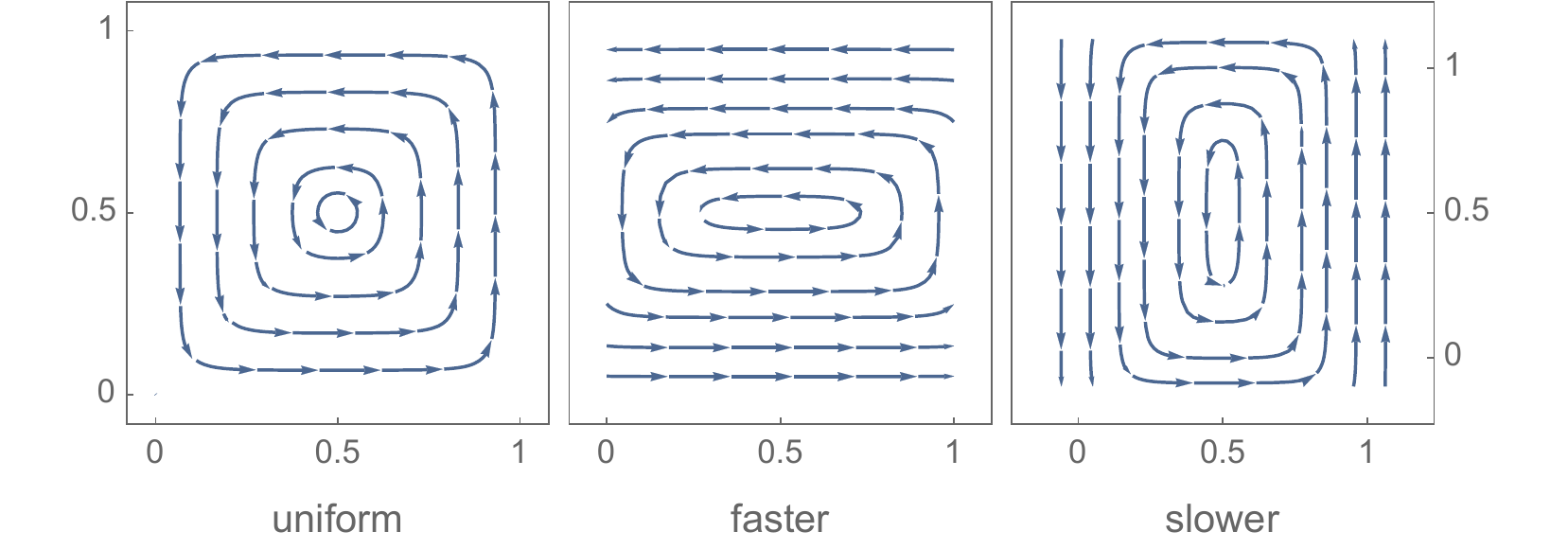}
  \caption{Gradient dynamics of the matching pennies game where agents learning have different learning rates. The vector field of the gradient dynamics are stretched along the faster agent's coordinate.}
  \label{fig:matching}
  \end{figure} 

The next example is again a  multi-agent policy gradient example in which there
are two players playing `matching pennies', a classic bimatrix game in which
agents have zero-sum costs associated with the matrices $(A,B)$ defined as
follows:
\[A=\begin{bmatrix} \ \ \ 1 & -1 \\ -1 & \ \ \ 1 \end{bmatrix},\quad B=\begin{bmatrix} -1 &
   \ \ \  1 \\ \ \ \ 1 & -1 \end{bmatrix}.\]
In particular, the players aim to minimize their respective costs $f_1(x,y) =
\pi(y)^T A \pi(x)$ and $f_2(x,y) = \pi(x)^T B \pi(y)$ where $\pi(x)$ is player
1's policy and $\pi(y)$ is player 2's policy. 
The class of policies the agents are optimizing over are the so-called `softmax'
policies defined by
\[\pi(z) = \left[\frac{e^{10z}}{e^{10z} + e^{10(1-z)}},\quad
\frac{e^{10(1-z)}}{e^{10z} + e^{10(1-z)}} \right],\]
and the update each player employs is a `smoothed best-response' which in
essence is a policy gradient update with respect to the softmax parameter and
each agents individual cost.
This game has been well studied in the game theory literature and we use it
illustrate the fact that non-uniform learning rates result in a warping of the
vector field associated with the agents' learning dynamics.

The mixed Nash equilibrium for this game is $(x^\ast,\ y^\ast) = (0.5,0.5)$, but
the Jacobian of the gradient dynamics at this fixed point is \[J(x^\ast,\
y^\ast)=\bmat{0 & 100 \\ -100 & 0}\] so that it has purely imaginary eigenvalues $\pm
100 i$, and therefore admits a limit cycle. 
Regardless, we can visualize the effects of non-uniform learning rates to the gradient dynamics in Figure~\ref{fig:matching}. We notice that the gradient flow stretches along the axes of the faster agent (the agent with a larger learning rate), and the fixed points of these dynamics remain constant.

  \subsection{Exploring the Effects of Non-uniform Learning Rates on the
  Learning Path}
  \label{subsec:torus}
  \begin{figure}[t!]
  \centering 
  \subfloat[][]{ 
      \label{fig:torus_roa}
      \includegraphics[width=.4\columnwidth]{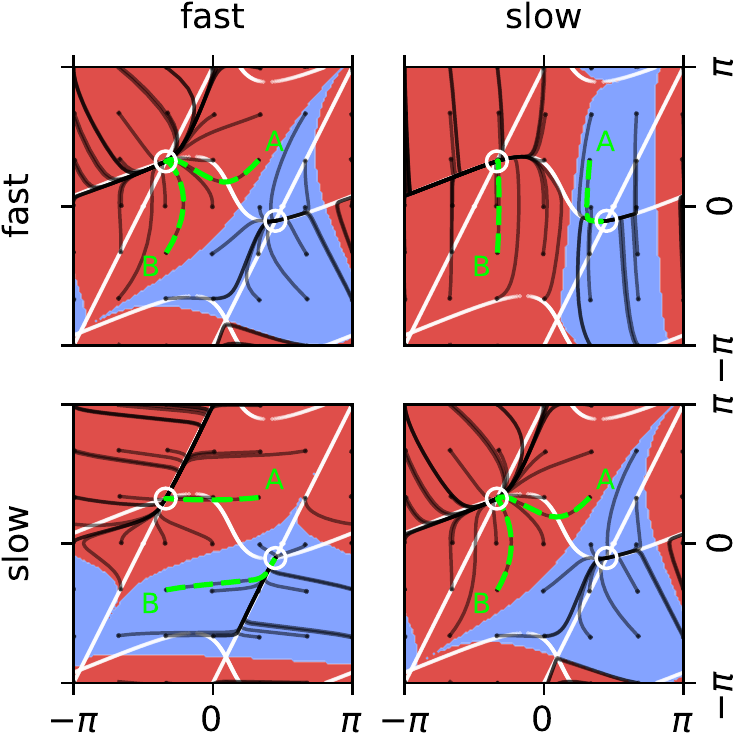}
  }
  \subfloat[][]{
    \label{fig:torus_stoch}
   $\qquad$ \includegraphics[width=.4\columnwidth]{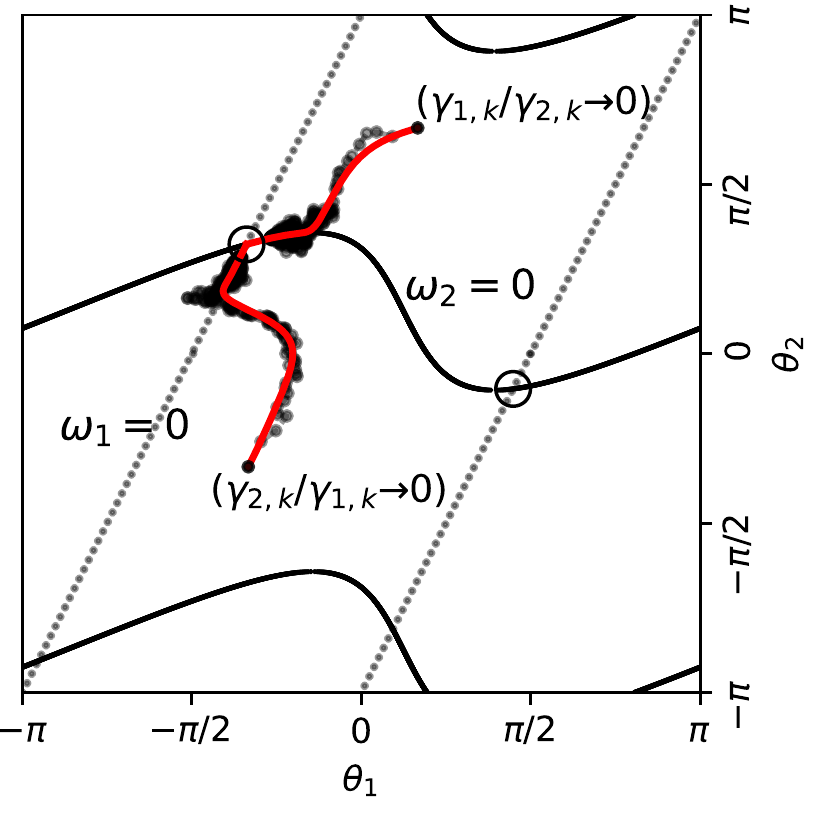}
  } 
  \caption{
    The effects of non-uniform learning rates on the path of convergence to the
equilibria. The zero lines for each player ($D_1f_1=0$ or $D_2f_2=0$) are
plotted as the diagonal and curved lines, and the two stable Nash equilibria as
circles (where $D^2_1 f_1 > 0$ and $D^2_2 f_2>0$). (a) In the deterministic
setting, the region of attractions for each equilibrium can be computed
numerically. Four scenarios are shown, with a combination of fast and slow
agents. The region of attractions for each Nash equilibrium are warped under
different learning rates. (b) In the stochastic setting, the samples (in black)
approximate the singularly perturbed differential equation (in red). Two initializations and learning rate configurations are plotted.}
  \label{fig:torus}
\end{figure} 
The examples presented so far all consider convergence (or non-convergence) to a
single equilibrium. In the following two examples, we investigate the effect of
non-uniform learning rates for more general non-convex settings in which there
are multiple equilibria.
The following example is a two-player game in which the agents' joint strategy
space is a torus.
That is, each player's strategy space is the unit circle $\mb{S}^1$. For each
$i\in\{1,2\}$, player $i$ has cost $f_i: \mathbb S^1 \times \mathbb S^1 \to \R$ given by 
\begin{equation*}
    f_i(\theta_i,\theta_{-i})=-\alpha_i\cos(\theta_i-\phi_i)+\cos(\theta_i-\theta_{-i})
\end{equation*}
  where $\alpha_i$ and $\phi_i$ are constants, and $\theta_i$ is player $i$'s choice
  variable. An interpretation of this game is that of a `location game' in which
  each player wishes to be near location $\phi_i$ but far from each other. This game has many applications including those which abstract nicely to coupled oscillators.

The game form---i.e., collection of individual gradients---is given by 
  \begin{equation}
      \omega(\theta_1, \theta_2) = \bmat{
          \alpha_1 \sin (\theta_1-\phi_1) -  \sin(\theta_1 - \theta_2) \\
          \alpha_2 \sin (\theta_2 -\phi_2) -  \sin(\theta_2 - \theta_1)
      },
  \end{equation}
  and the game Jacobian is composed of terms
  $\alpha_i\cos(\theta_i-\phi_i)-\cos(\theta_i-\theta_{-i})$, $i=1,2$ on the
  diagonal and $\cos(\theta_i-\theta_{-i})$, $i=1,2$ on the off-diagonal.

 The Nash equilibria of this game occur where $\omega(\theta_1, \theta_2)=0$ and where the diagonals of the game Jacobian are positive.
  The game has multiple Nash equilibria. We visualize the warping of the region of attraction of these equilibria under different learning rates, and the affinity of the ``faster'' player to its own zero line.

  In this example, we use constants $\phi = (0,\ \pi/8)$ and $\alpha = (1.0,\
  1.5)$. The joint strategy space can be viewed as a non-convex smooth manifold
  via an equivalence relationship,
  or equivalently, as players choosing $\theta_i\in\mb{R}$. There are two
 Nash equilibria, situated at $(-1.063,\ 1.014)$ and $(1.408,\ -0.325)$.
  These equilibria happen to also be stable differential Nash, and thus we
  expect the gradient dynamics to converge to them if initialized in the region
  of attraction. Which equilibrium it converges to, however, depends on the initialization and learning rates of agents.

  To investigate how non-uniform learning rates affect the agents' convergence to the two equilibria, we simulate agents learning at different rates, one fast and one slow. The fast agent's learning rate is set to $\gamma_1=0.171$ and the slow $\gamma_2 = 0.017$. Figure~\ref{fig:torus_roa} shows the trajectory of agents' learned strategies. Each of the four squares depicts the full strategy space on the torus from $-\pi$ to $\pi$ for both agents' actions, with $\theta_1$ on the $x$-axis and $\theta_2$ on the $y$-axis. The labels ``fast'' and ``slow'' indicate the learning rate of the corresponding agent. For example, in the bottom left square, agent 1 is the fast agent and agent 2 is the slow agent. Hence, the non-uniform update equation for that square becomes 
  $\theta_{k+1} = \theta_{k} - \mathrm{diag}(\gamma_1, \gamma_2) \omega(\theta_k).$

  The white lines indicate the points $x$ such that  $\omega_i(x)= 0$, and the
  intersection of the white lines indicate points $x$ such that $\omega(x)=0$.
  The two intersections marked as circles are the stable differential Nash
  equilibria. The unmarked intersections are either saddle points or other
  unstable equilibria. The black lines show different paths of the update
  equations under the non-uniform update equation, with initial points selected
  from a equally spaced $7\times 7$ grid. We highlight two paths in green
  (labeled A and B) which begin at $(\pi/3,\ \pi/3)$ and $(-\pi/3, -\pi/3)$.
  
  In the case where agents both learn at the same rate, $(\gamma_1,
  \gamma_1)$ and $(\gamma_2, \gamma_2)$,
  paths A and B both converge to the Nash equilibrium at $(-1.063,\ 1.014)$.
  However, when agents learn at different rates, the equilibrium to which the
  agents converge to, as well as the learning path, is no longer the same even
  starting at the same initial points.  This phenomena can also be captured by
  displaying the region of attraction for both Nash equilibria. The red region
  corresponds to initializations that will converge to the equilibrium contained
  in the red region (again indicated by a white circle). Analogously, the blue
  region corresponds to the region of attraction of the other equilibria.

 To provide an example of the stochastic setting in which agents have an
 unbiased estimator of their individual gradients, we choose learning rates
 according to Assumption~\ref{ass:lip}. In particular, we choose scaled learning
 rates $\gamma_{2,k} = \frac{1}{1+k\log(k+1)}$ and
 $\gamma_{1,k} = \frac{1}{1+k}$ such that
 $\gamma_{2,k}/\gamma_{1,k} \to 0$ as $k \to \infty$.
 Figure~\ref{fig:torus_stoch} shows the learning paths in this setting
 initialized at two different points, each with flipped learning rate
 configurations. The sample points approximate the singularly perturbed
 differential equation (shown in red) described in
 Section~\ref{subsec:stochastic_multiplelearn}.

  In both deterministic and stochastic settings, we observe the affinity of the
  faster agent to its own zero line. For example, the bottom left square (in
  Figure~\ref{fig:torus_roa}) and bottom left path (in
  Figure~\ref{fig:torus_stoch}) both have agent 1 as the faster agent, and the
  learning paths both tend to arrive to the line $\omega_1\equiv 0$ before finally converging to the Nash equilibrium. An interpretation of this is that the faster agent tries to be situated at the bottom of the ``valley'' of its own cost function. The faster agent tends to be at its \emph{own} minimum while it waits for the slower agent to change its strategy. As a Stackelberg interpretation, where there are followers and leaders, the slower agent would be the leader and faster agent the follower. In a sense, the slower agent has an advantage.

  \subsection{Multi-agent control and collision avoidance}

\begin{figure}[]
\centering  \subfloat[][]{
      \includegraphics[width=0.225 \columnwidth]{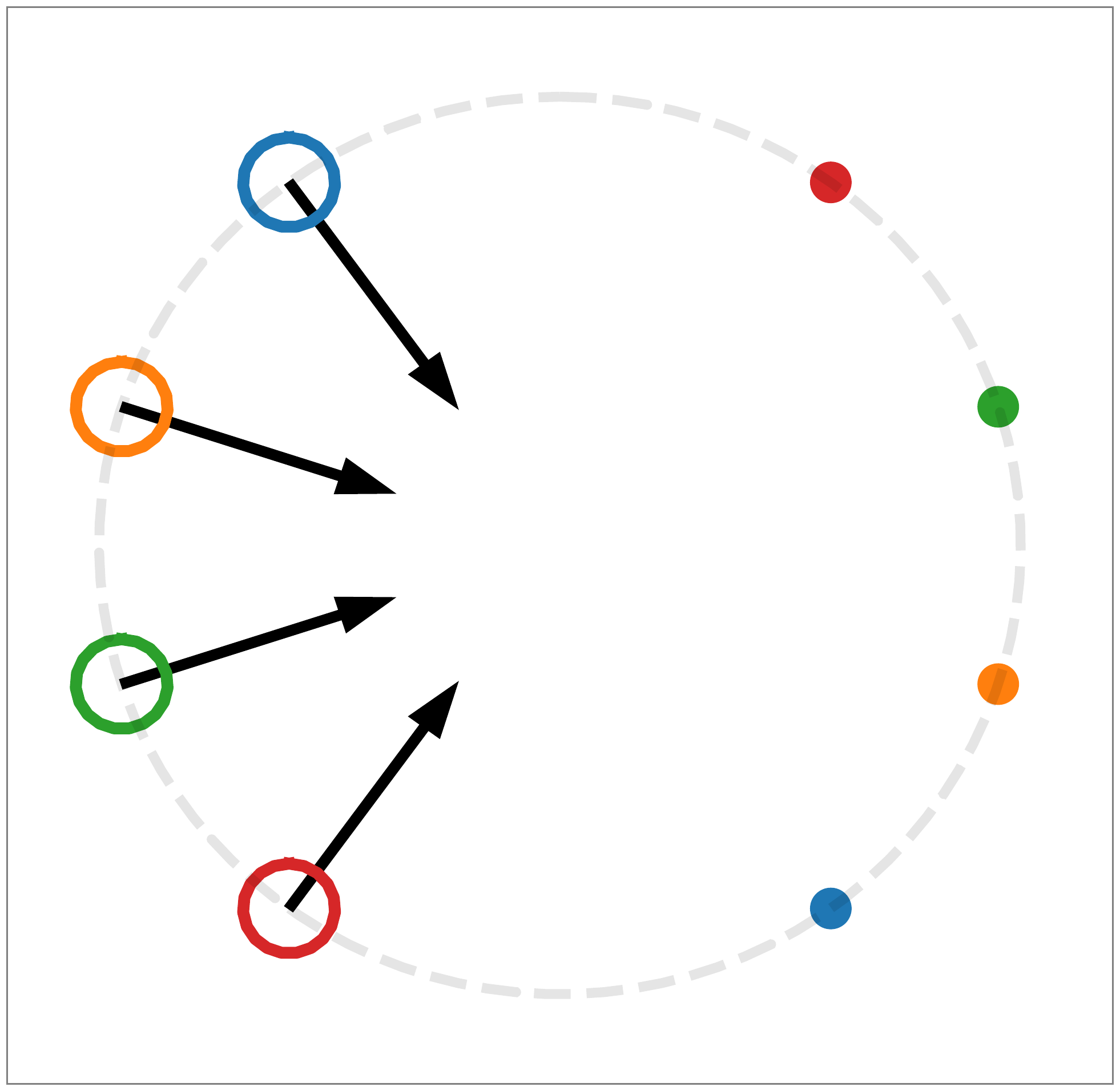}
      }
  \subfloat[][]{
      \includegraphics[width=0.225 \columnwidth]{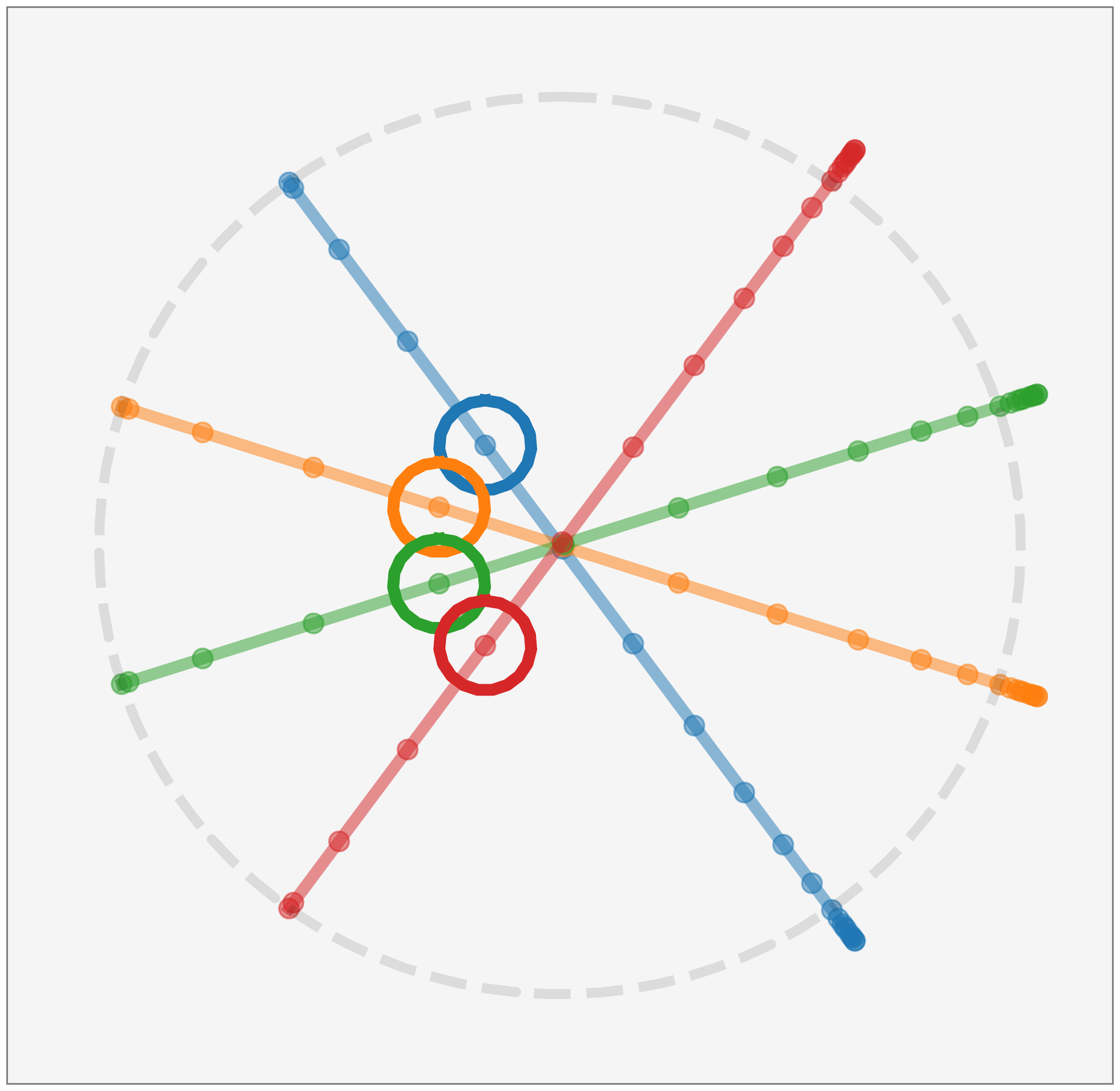}
      } 
  \subfloat[][]{
      \includegraphics[width=0.225 \columnwidth]{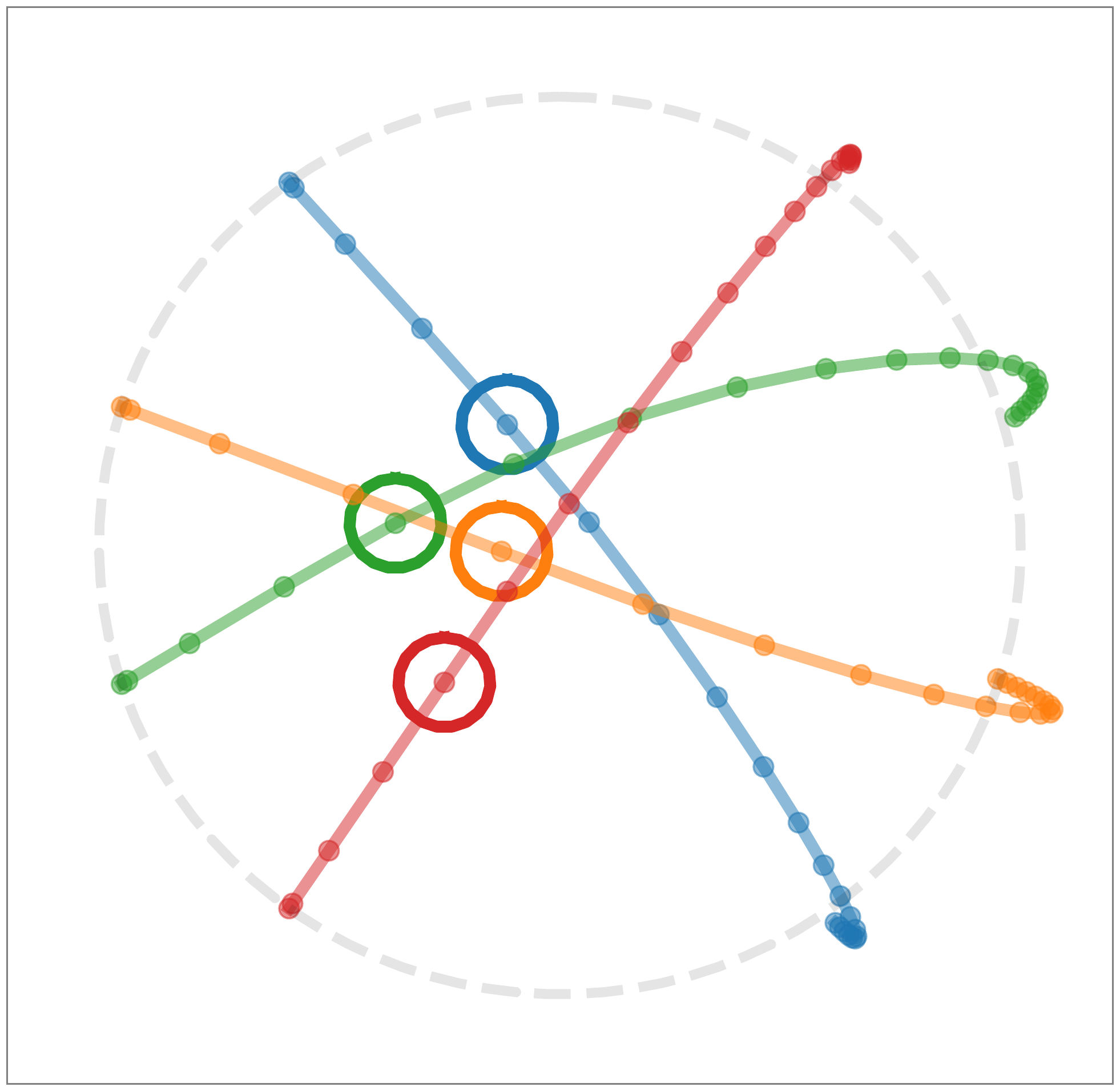} 
  }
  \subfloat[][]{
      \includegraphics[width=0.225 \columnwidth]{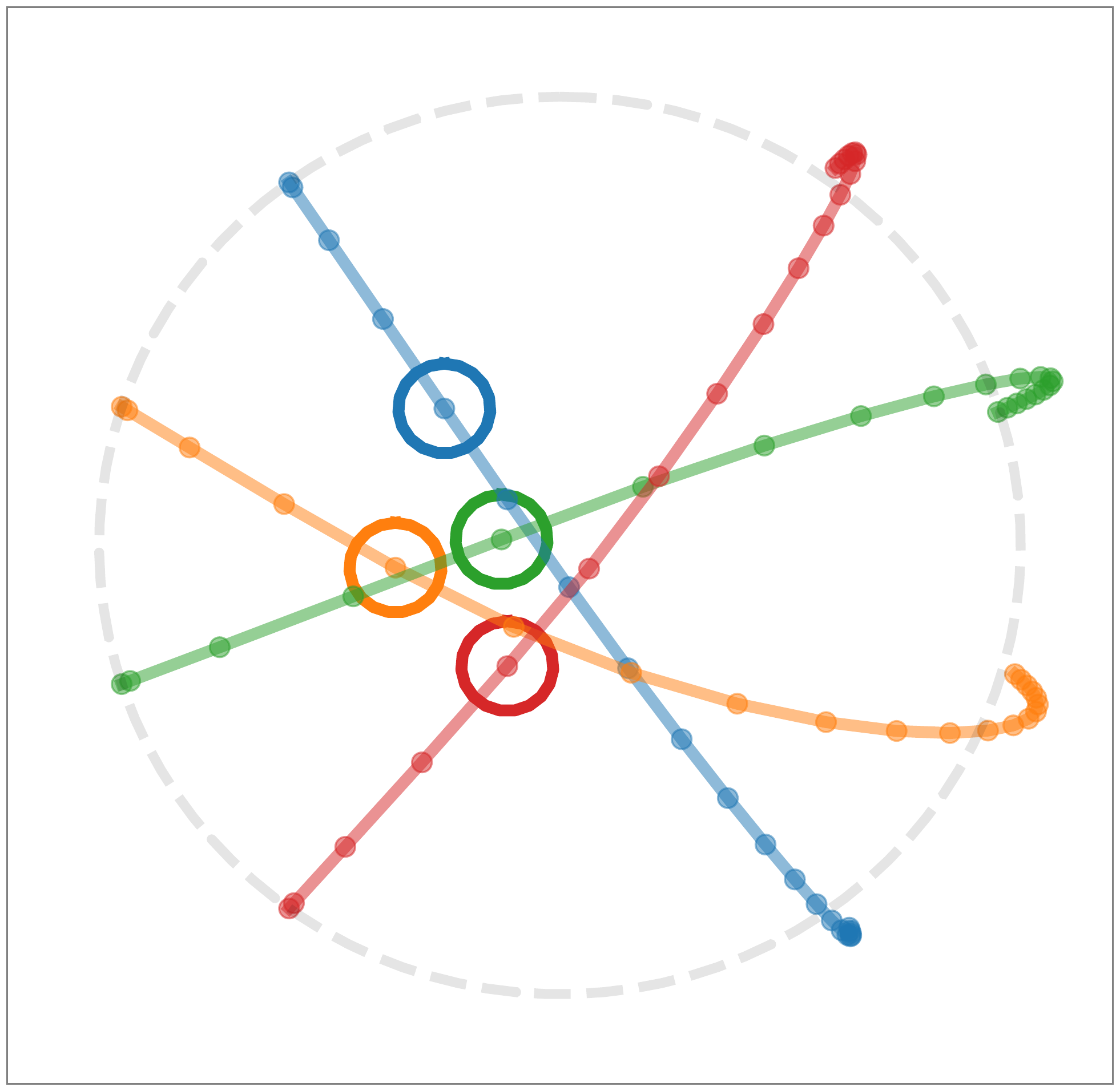}   
  } 
\caption{Minimum-fuel particle avoidance control example. 
(a) Each particle seeks to reach the opposite side of the circle using minimum fuel while avoiding each other. The circles represent the approximate boundaries around each particle at time $t=5$.
(b) The joint strategy $x=({\bf u}_1, \cdots, {\bf u}_4)$ is initialized to the minimum fuel solution ignoring interaction between particles. 
(c) Equilibrium solution achieved by setting the blue agent to have a slower learning rate. (d) Another equilibrium, where the red agent has the slower learning rate. }

\label{fig:bubble}
\end{figure}
The final example presents a practical use case for the gradient-based update.
Consider a non-cooperative game between four collision-avoiding agents where
they seek to arrive at a destination with minimum fuel while avoiding each
other. We show that the scaling between agents' learning rates dictates the
equilibrium solution to which they converges. This can be useful in designing
non-cooperative open-loop controllers where agents may choose to learn slower in
order to deviate less from their initial plan, perhaps in an attempt to incur
less `risk'.

Suppose there are four collision-avoiding particles traversing across a unit
circle. Each particle follows discrete-time linear dynamics \[z_i(t+1) = Az_i(t)
+ Bu_i(t)\] for $t=1,\cdots,N$ where 
\[A=\begin{bmatrix}
    I & h I \\
    0 & I
\end{bmatrix} \in \R^{4\times4}, \  %
B=\begin{bmatrix}
    h^2 I \\
    h I
\end{bmatrix} \in \R^{4\times2},\]
$I$ is the identity matrix, and $h=0.1$. These dynamics represent a typical
discretized version of the continuous dynamics $\ddot r_i = u_i$ in which $u_i\in \R^2$
represents a force vector used to accelerate the particle, and the state
$z_i=[r_i,\dot r_i]$ represents the particles position and velocity.   
Let ${\bf u}_i$ be the concatenated vector of control vectors for player $i$ for
all time---i.e., ${\bf u}_i = (u_i(1), \cdots, u_i(N))$ and let ${\bf u}=({\bf
u}_1, \cdots, {\bf u}_n)$.
Each particle $i$ aims to minimize a cost defined by 
\begin{align*}
    J_i({\bf u}) &=\sum_{t=1}^N \|u_i(t)\|_R^2 + 
                        \sum_{t=1}^{N+1} \| z_i(t) - \bar{z_i} \|_Q^2 + \sum_{j\neq i} \sum_{t=1}^{N+1} \rho e^{-\sigma \|z_i(t) - z_j(t)\|_S^2}
\end{align*}
where $\|\cdot \|_P$ denotes the quadratic norm---i.e., $\|z\|_P^2 = z^T P z$ with $P$ positive semi-definite. The first two terms of the cost correspond to the minimum fuel objective and quadratic cost from desired final state $\bar z_i$, a typical setup for optimal control problems. We use $R=\diag(0.1,0.1)$ and $Q=\diag(1,1,0,0)$. The final term of the cost function is the sum of all pairwise interaction terms
between the particles, modeled after the shape of a Gaussian which encodes smooth boundaries around the particles. We use 
constants $\rho=10$ and $\sigma=100$. 

Figure~\ref{fig:bubble} (a) visualizes the problem setup. Each particle's
initial position $z_i(0)$ is located on the left side of a unit circle; they are
separated by $\pi/5$, and their desired final positions, $\bar{z_i}$ for each
$i\in\{1, \ldots, 4\}$, are located directly opposite. The particles begin with
zero velocity and must solve for a minimum control solution that also avoids
collision with other particles as described by the objectives $J_i$ for each
$i$.

To initialize the gradient-based learning algorithms in the game setting, we
compute the 
optimal solution for each agent
ignoring the pairwise interaction terms, shown in Figure~\ref{fig:bubble} (b).
This can be computed using classical discrete-time LQR methods or by gradient
descent. Then, using this solution as the intialization for the game setting,
each agent descends their own gradient, i.e.
\[{\bf
u}_{i,k+1} = {\bf u}_{i,k} - \gamma_i D_i J_i({\bf u}),\] 
with different learning rates $\gamma_i$.
Just as the previous example shows, the relative learning rates of agents warp the region of attraction for the multiple equilibria. If we allow the red agent to learn slower, then the learning process converges to the equilibria shown in 
Figure~\ref{fig:bubble} (c), whereas if the blue agent learns quicker, then we converge to Figure~\ref{fig:bubble} (d). 
 Hence, all else being equal, the learning rates adopted by players greatly impact the equilibrium to which they converge.

\section{Discussion and Future Work}
\label{sec:discussion}
We analyze the convergence of gradient-based learning for non-cooperative agents
with continuous costs. We leverage existing dynamical systems theory and
stochastic approximation literature to provide convergence guarantees for
agents that learn myopically---that is, only using information about their own
gradient $D_i f_i$ to update their strategy. We provide guarantees for  the
case where agents 
are assumed to have oracle access to $D_if_i$ and the case where they have sufficient information
to compute an unbiased estimator. We also study the effects of non-uniform
learning rates. 

By preconditioning the gradient dynamics by $\Gamma$, a diagonal matrix where the diagonals represent the agents' learning rates, we can begin to understand how a changing learning rate relative to others can change the properties of the fixed points of the dynamics. Moreover, players do not know how a change in others' strategies affects its own cost ($D_j f_i$ where $j\neq i$). A possible extension to this paper is to develop update schemes that use this to provide more robust convergence guarantees for full information continuous games.
Different learning rates amongst agents also affects the region of attraction of
the game, hence starting from the same initial condition, agents may converge to a different equilibria. Agents may use this to their benefit, as shown in the last example. Such insights into the learning behavior of agents will be useful for providing guarantees on the design of control or incentive policies to coordinate agents. We also show through numerical examples that, counterintuitively, if an agent decides to learn slower, a stable differential Nash equilibrium can go unstable, resulting in learning dynamics that do not converge to Nash.

Beyond the the effects of learning rates, there are a number of avenues for
future inquiry. For instance, the results as stated apply to
continuous games with Euclidean strategy spaces. An interesting avenue to
pursue is the study of learning in games where the agents decision spaces are constrained sets or
Riemannian manifolds. The latter arises in a number of robotics applications and
in this case, the update rule will need to be modified
by the appropriately defined retraction such as
$x_{k+1}=\exp_{x_k}(\gamma_k(\omega(x_k)))$~\cite{shah:2017aa}. The former
arises in a variety of applications where the learning rules are abstractions of
agents learning in, e.g., physically constrained environments. The update rule
in this case 
will also need to be defined in terms of the appropriate proximal map thereby leading
to potentially non-smooth dynamics~\cite{bokar:2008aa,kushner:2003aa} which is
even more challenging in the stochastic setting. Yet, such extensions will lead
to a framework and set of analysis tools that apply to a broader class of
multi-agent learning algorithms.

While we present the work in the context of gradient-based learning in games, there is nothing that precludes the results from applying to update rules in other frameworks. Our results will apply to many other settings where agents myopically update their decision using a process of the form $x_{k+1} = x_k - \Gamma g(x_k)$. In this paper, we consider the special case where $g \equiv [D_1f_1 \cdots D_nf_n]$. In the stochastic setting, variants of multi-agent Q-learning conform to this setting since Q-learning can be written as a stochastic approximation update. 

Finally, as pointed out in~\cite{mazumdar:2018aa}, not all critical points of the dyanamics $\dot x = -\omega(x)$ that are attracting are necessarily Nash equilibria; one can see this simply by constructing a Jacobian with positive eigenvalues with at least one $D^2_{i}f_i$ with a non-positive eigenvalue. Understanding this phenomena will help us develop computational techniques to avoid them. Recent work has explored this in the context of zero-sum games~\cite{mazumdar:2019aa}, requiring coordination amongst the learning agents. However, when our objective is to study the learning behavior of autonomous agents seeking an equilibrium, an alternative perspective is needed.

\appendix

\section{Proofs}
\subsection{Deterministic Setting}
\label{app:proofs_deterministic}

The following proof follows nearly the same proof as the main result
in~\cite{argyros:1999aa} with a few minor modifications in the conclusion; we provide it here for posterity.
\begin{proof}[Proof~Proposition~\ref{prop:nonuniformone}]
    Since $\|I-\Gamma D\omega(x)\|<1$ for each $x\in B_{r_0}(x^\ast)$, as stated
    in the proposition statement,
   there exists $0<r'< r''<1$ such that $\|I-\Gamma
    D\omega(x)\|\leq r'<r''<1$ for all $x\in B_r(x^\ast)$. Since 
    \[\lim_{x\rar x^\ast}\|R(x-x^\ast)\|/\|x-x^\ast\|=0,\]
    for $0<1-r''<1$, there exists $\tilde{r}>0$ such that
    \[\|R(x-x^\ast)\|\leq (1-r'')\|x-x^\ast\|, \ \ \forall \ x\in B_{\tilde{r}}(x^\ast).\]
    As in the proposition statement, let $r$ be the largest, finite such $\tilde{r}$.
    Note that for arbitrary $c>0$, there exists $\tilde{r}>0$ such that the
    bound on $\|R(x-x^\ast)\|$ holds; hence, we choose $c=1-r''$ and find the
    largest such $\tilde{r}$ for which the bound holds.
  Combining the above bounds with the definition of $g$, we have that
  \[\|g(x)-g(x^\ast)\|\leq (1-\delta)\|x-x^\ast\|, \ \ \forall \ x\in
    B_{r^\ast}(x^\ast)\]
    where $\delta=r''-r'$ and $r^\ast=\min\{r_0,r\}$.  Hence, applying the result
    iteratively, we have that
    \[\|x_t-x^\ast\|\leq (1-\delta)^t\|x_0-x^\ast\|, \ \ \forall \ x_0\in
    B_{r^\ast}(x^\ast).\]
    Note that $0<1-\delta<1$. Using the approximation
    $1-\delta<\exp(-\delta)$, we have that
    \[\|x_T-x^\ast\|\leq \exp(-T\delta)\|x_0-x^\ast\|\] so that $x_t\in
    B_\vep(x^\ast)$ for all $t\geq T=\lceil \delta^{-1} \log(r^\ast/\vep)\rceil$.

\end{proof}
As noted in the remark, a similar result holds under the relaxed assumption that
$\rho(I-\Gamma D\omega(x))<1$ for all $x\in B_{r_0}(x^\ast)$.
To see this, we first note that $\rho(I-\Gamma D\omega(x))<1$ implies there
exists $c>0$ such that $\rho(I-\Gamma D\omega(x))\leq c<1$. Hence, given any
$\epsilon>0$, there is a norm on $\mb{R}^d$ and a $c>0$ such that $\|I-\Gamma D\omega\|\leq
c+\epsilon<1$ on $B_{r_0}(x^\ast)$~\cite[2.2.8]{ortega:1970aa}. Then, we can apply the same argument as
above using $r'=c+\vep$.

\subsection{Stochastic Setting}
\label{app:proofs_stocahstic}
A key tool used in the finite-time two-timescale analysis is the nonlinear
variation of constants formula of Alekseev~\cite{alekseev:1961aa},
\cite{borkar:2018aa}. 
\begin{theorem}
    Consider a differential equation 
    \[\dot{u}(t)=f(t,u(t)), \ t\geq 0,\]
    and its perturbation
    \[\dot{p}(t)=f(t,p(t))+\tilde{f}(t,p(t)), \ t\geq 0\]
    where $f,\tilde{f}:\mb{R}\times \mb{R}^d\rar \mb{R}^d$, $f\in C^1$, and
    $\tilde{f}\in
    C$. Let $u(t,t_0,p_0)$ and $p(t,t_0,p_0)$ denote the solutions of the above
    nonlinear systems for $t\geq t_0$ satisfying
    $u(t_0,t_0,p_0)=p(t_0,t_0,p_0)=p_0$, respectively. Then,
    \begin{align*}
        p(t,t_0,p_0)&=u(t,t_0,p_0)+\int_{t_0}^t\Phi(t,s,p(s,t_0,p_0)) \tilde{f}(s,p(s,t_0,p_0))\
        ds, \ t\geq t_0
    \end{align*}
    where $\Phi(t,s,u_0)$, for $u_0\in \mb{R}^d$, is the fundamental matrix of
    the linear system 
    \begin{equation}
        \dot{v}(t)=\frac{\partial f}{\partial u}(t,u(t,s,u_0))v(t), \ t\geq s
        \label{eq:vdot}
    \end{equation}
    with $\Phi(s,s,u_0)=I_d$, the $d$--dimensional identity matrix.
    \label{thm:alekseev}
\end{theorem}

Consider a locally asymptotically stable differential Nash equilibrium
$x^\ast=(
\lambda(x_2^\ast),x_2^\ast)\in X$ and let $B_{r_0}(x^\ast)$ be an ${r}_0>0$ radius ball around $x^\ast$ contained in the
region of attraction.
Stability implies that the Jacobian $J(
\lambda(x_2^\ast),x_2^\ast)$ is positive definite and by the converse Lyapunov
theorem~\cite[Chapter~5]{sastry:1999aa} there
exists local Lyapunov functions for the dynamics $\dot{x}_2(t)=-\tau
D_2f_2(\lambda(x_2(t)),x_2(t))$ and for the dynamics
$\dot{x}_1(t)=-D_1f_1(x_1(t),
x_2)$, for each fixed $x_2$. In particular, there exists a local Lyapunov
function $V\in C^1(\mb{R}^{d_1})$ with $\lim_{\|x_2\|\uparrow
\infty}V(x_2)=\infty$, and $\la \nabla V(x_2), D_2f_2(\lambda(x_2),x_2)\ra<0$
for $x_2\neq
x_2^\ast$. 
For $r>0$, let $V^r=\{x\in \text{dom}(V):\ V(x)\leq r\}$. Then, there is also 
$r>r_0>0$ and $\epsilon_0>0$ such that for $\epsilon<\epsilon_0$,
\[\{x_2\in \mb{R}^{d_2}|\ \|x_2-x_2^\ast\|\leq \epsilon\}\subseteq V^{r_0}\subset
\mc{N}_{\epsilon_0}(V^{r_0})\subseteq V^r\subset \text{dom}(V)\] where
$\mc{N}_{\epsilon_0}(V^{r_0})=\{x\in \mb{R}^{d_2}|\ \exists x'\in V^{r_0}\
\text{s.t.} \|x'-x\|\leq \epsilon_0\}$. An analogously defined $\tilde{V}$
exists for the dynamics $\dot{x}_1$ for each fixed $x_2$. 

For now, fix $n_0$ sufficiently large; we specify this a bit further down. Define the
event \[\mc{E}_\tind=\{\bar{x}_1(t)\in V^r\ \forall t\in [\tilde{t}_{\no},
\tilde{t}_\tind]\}\] where
\[\bar{x}_1(t)=x_{1,k}+\frac{t-\tilde{t}_k}{\gamma_{1,k}}(x_{1,k+1}-x_{1,k})\] are
linear interpolates defined for $t\in (\tilde{t}_k, \tilde{t}_{k+1})$ with
$\tilde{t}_{k+1}=\tilde{t}_k+\gamma_{1,k}$ and $\tilde{t}_0=0$. The basic idea
of the proof is to leverage Alekseev's formula (Theorem~\ref{thm:alekseev}) 
to bound the difference between the linearly
    interpolated trajectories (i.e., \emph{asymptotic psuedo-trajectories}) and
        the flow of the corresponding limiting differential equation on each
        continuous time interval between each of the successive iterates $k$ and
        $k+1$ by a number that decays asymptotically. Then, for large enough
        $k$, a union bound is used over all the remaining time intervals
        to construct a concentration bound. This is done first for fast player
        (i.e. player 1),
    to show that $x_{1,k}$ tracks  $\lambda(x_{2,k})$, and
    then for the slow player (i.e., player 2). %

Following~\citet{borkar:2018aa}, we can express the linear interpolates for any
$\tind\geq \no$ as
    $\bar{x}_1(\tilde{t}_{\tind+1})\textstyle=\bar{x}_1(\tilde{t}_{\no})-\sum_{\ell=\no}^n\gamma_{1,\ell}(D_1f_1(x_{\ell})+w_{1,\ell+1})$
where
\[\gamma_{1,\ell}D_1f_1(x_{\ell})=\int_{\tilde{t}_\ell}^{\tilde{t}_{\ell+1}}D_1f_1(\bar{x}_1(\tilde{t}_\ell),x_{2,\ell})\]
and similarly for the $w_{1,\ell+1}$ term.
Adding and subtracting
$\int_{\tilde{t}_{n_0}}^{\tilde{t}_{n+1}}D_1f_1(\bar{x}_1(s), x_{2}(s))$, Alekseev's formula 
can be applied to get 
\begin{align*}
    \bar{x}_1(t)&=x_1(t)+\Phi_1(t,s,\bar{x}_1(\tilde{t}_{\no}),x_2(\tilde{t}_{\no}))(\bar{x}_1(\tilde{t}_{\no})-x_1(\tilde{t}_{\no}))+\int_{\tilde{t}_{\no}}^t\Phi_2(t,s,\bar{x}_1(s),x_2(s))\zeta_1(s)\ ds
\end{align*}
where $x_2(t)\equiv x_2$ is constant (since $\dot{x}_2=0$), $x_1(t)=\lambda(x_2)$, 
\[\zeta_1(s)=-D_1f_1(\bar{x}_1(\tilde{t}_k),x_2(\tilde{t}_k))+D_1f_1(\bar{x}_1(s),x_2(s))+w_{1,k+1},\]
 and  where for $t\geq s$, $\Phi_1(\cdot)$ satisfies linear
 system 
 \begin{equation*}
     \dot{\Phi}_1(t,s,x_{0})=J_1(x_1(t),x_2(t))\Phi_1(t,s,x_{0}),
 \end{equation*}
 with initial data $\Phi_1(t,s,x_{0})=I$ and $x_0=(x_{1,0},x_{2,0})$ and where
 $J_1$ the Jacobian of
$-D_1f_1(\cdot,x_2)$. 

Given that $x^\ast=(\lambda(x_2^\ast), x_2^\ast)$ is a stable differential
Nash equilibrium, $J_1(x^\ast)$ is positive definite. Hence,  as in
\cite[Lemma~5.3]{thoppe:2018aa}, we can find
$M$,
$\kappa_1>0$ such that for $t\geq s$, $x_{1,0}\in V^r$,
$\|\Phi_1(t,s,x_{1,0},x_{2,0})\|\leq Me^{-\kappa_1(t-s)}$; this result follows
from 
standard results on stability of linear systems (see, e.g., \citet[\S7.2,
 Theorem~33]{callier:1991aa})  along with 
 a bound on
 \[\int_{s}^t\|D^2_1f_1(x_{1},x_{2}(\tau,s,\tilde{x}_0))-D_1^2f_1(x^\ast)\|d\tau\]
 for $\tilde{x}_0\in V^r$ (see, e.g.,~\cite[Lemma~5.2]{thoppe:2018aa}).

Consider $z_k=\lambda(x_{2,k})$---i.e., where $D_1f_1(x_{1,k},x_{2,k})=0$. Then,
using a Taylor expansion of the implicitly defined $\lambda$, we get
\begin{equation}
    z_{k+1}=z_k+D \lambda(x_{2,k})(x_{2,k+1}-x_{2,k})+\delta_{k+1}
    \label{eq:taylor}
\end{equation}
where $\|\delta_{k+1}\|\leq L_{r}\|x_{2,k+1}-x_{2,k}\|^2$ is the error from
the remainder terms. Plugging in $x_{2,k+1}$, we have
\begin{align*}
    z_{k+1}&=z_k+\gamma_{1,k}(-D_1f_1(z_k,x_{2,k})+\tau_k\lambda(x_{2,k})(w_{2,k+1}-D_2f_2(x_{1,k},x_{2,k}))+\gamma_{1,k}^{-1}\delta_{k+1}).
\end{align*}
The terms after $-D_1f_1$ are $o(1)$, and hence asymptotically negligible, so
that this $z$ sequence tracks dynamics as $x_{1,k}$. 
We show that with high
probability, they
asymptotically contract to one another. 
Define constant
$H_{\no}=(\|\bar{x}_1(\tilde{t}_{\no}-x_1(\tilde{t}_{\no})\|+\|\bar{z}(\tilde{t}_{\no})-x_1(\tilde{t}_{\no})\|)$
and
\begin{align*}
    S_{1,\tind}&=\sum_{\ell=\no}^{\tind-1}\Big(\int_{\tilde{t}_\ell}^{\tilde{t}_{\ell+1}}\Phi_1(\tilde{t}_\tind,s,\bar{x}_1(\tilde{t}_\ell),x_{2}(\tilde{t}_\ell))
    ds\Big) w_{2,\ell+1}.
\end{align*}
Moreover, let $\tau_k=\gamma_{2,k}/\gamma_{1,k}$. 
\begin{lemma}
    For any $\tind\geq \no$,
     there exists $K>0$ such that
    \begin{align*}
       \|x_{1,\tind}-z_\tind\|\leq&
       K\Big(\|S_{1,\tind}\|+e^{-\kappa_1(\tilde{t}_\tind-\tilde{t}_{\no})}H_{\no}+\sup_{\no\leq \ell\leq \tind-1}\gamma_{1,\ell}+\sup_{\no\leq
            \ell\leq
        \tind-1}\gamma_{1,\ell}\|w_{1,\ell+1}\|^2\\
        &+\sup_{\no\leq \ell\leq \tind-1}\tau_\ell+\sup_{\no\leq \ell\leq
        \tind-1}\tau_\ell\|w_{2,\ell+1}\|^2\Big)
    \end{align*}
    conditioned on ${\mc{E}}_\tind$.
        \label{lem:defK}
\end{lemma}

In order to construct a high-probability bound for $x_{2,k}$, we need a similar bound as in Lemma~\ref{lem:defK} can be
constructed for $x_{2,k}$. 
Define the event
$\hat{\mc{E}}_\tind=\{\bar{x}_2(t)\in V^{{r}}\ \forall t\in [\hat{t}_{\no},
\hat{t}_\tind]\}$ where
$\bar{x}_2(t)=x_{2,k}+\frac{t-\hat{t}_k}{\gamma_{2,k}}(x_{2,k+1}-x_{2,k})$ is
the linear interpolated points between the samples $\{x_{2,k}\}$,
$\hat{t}_{k+1}=\hat{t}_k+\gamma_{1,k}$, and $\hat{t}_0=0$. Then as above,
Alekseev's formula can again be applied to get 
\begin{align*}
    \bar{x}_2
    &(t)=x_2(t,\hat{t}_{\no},x_2(\hat{t}_{\no}))+\Phi_2(t,\hat{t}_{\no},
    \bar{x}_2(\hat{t}_{\no}))(\bar{x}_2(\hat{t}_{\no})-x_2(\hat{t}_{\no}))+\int_{\hat{t}_{\no}}^t\Phi_2(t,s,\bar{x}_2(s))\zeta_1(s)\
    ds
\end{align*}
where $x_2(t)\equiv x_2^\ast$,
\begin{align*}
    \zeta_1(s)&=D_2f_2(\lambda(x_{2,k}),x_{2,k})-D_2f_2(\lambda(\bar{x}_2(s)),\bar{x}_2(s))+D_2f_2(x_k)-D_2f_2(\lambda(x_{2,k}),x_{2,k})+w_{2,k+1},
\end{align*}
and $\Phi_2$ is the solution to a linear system
with dynamics $J_2(\lambda(x_2^\ast),x_2^\ast)$, the Jacobian of
$-D_2f_2(\lambda(\cdot),\cdot)$, and with initial data $\Phi_2(s,s,x_{2,0})=I$. This
linear system, as above, has bound $\|\Phi_2(t,s,x_{2,0})\|\leq
M_2e^{\kappa_2(t-1)}$ for some $M_2,\kappa_2>0$.
Define
\begin{align*}
    S_{2,\tind}&=\sum_{\ell=\no}^{\tind-1}\Big(\int_{\hat{t}_{\ell}}^{\hat{t}_{\ell+1}}\Phi_2(\hat{t}_\tind,s,\bar{x}_2(\hat{t}_\ell))ds\Big)w_{2,\ell+1}.
\end{align*}
\begin{lemma}
    For any $\tind\geq \no$, there exists $\bar{K}>0$ such that
    \begin{align*}
        \|\bar{x}_2(\hat{t}_\tind)-x_2(\hat{t}_\tind)\|\leq&\textstyle
        \bar{K}\big(\|S_{2,\tind}\|+\sup_{\no\leq \ell\leq \tind-1}\|S_{1,\ell}\|\textstyle+\sup_{\no\leq \ell\leq \tind-1}\gamma_{1,\ell}+\sup_{\no\leq
            \ell\leq
        \tind-1}\gamma_{1,\ell}\|w_{1,\ell+1}\|^2\notag\\
        &\textstyle+\sup_{\no\leq \ell\leq \tind-1}\tau_\ell+\sup_{\no\leq \ell\leq
        \tind-1}\tau_\ell\|w_{2,\ell+1}\|^2\textstyle+e^{\kappa_2(\hat{t}_\tind-\hat{t}_{\no})}\|\bar{x}_2(\hat{t}_{\no})-x_2(\hat{t}_{\no})\|\notag\\
        &\textstyle+\sup_{\no\leq
        \ell\leq \tind-1}\tau_kH_{\no}\big)
    \end{align*}
     conditioned on $\tilde{\mc{E}}_\tind$.
    \label{lem:defbarK}
\end{lemma}
Using the above lemmas, we can get the desired guarantees on $x_{1,k}$ and
$x_{2,k}$ as in~\cite{borkar:2018aa}.

\section{Additional Examples}
\label{app:examples}
In this appendix, we include additional examples and information about examples
contained in the main body of the text.
\subsection{LQ game system parameters}
\label{app:lqr}

The following are the system parameters and resulting Nash feedback gains
computed using the coupled Riccatti equations:
\begin{gather*}
    A = \bmat{0.402 & 1.037 & -0.565 & 0.115 \\
        -0.021 & -0.990 & -0.584 & 0.457 \\
        0.377 & 1.105 & 0.698 & 1.192 \\
        -0.177 & -0.332 & 0.237 & -0.286 
        },\ 
    B_1 = \bmat{1 \\
        1 \\
        0 \\
        0 },\ 
    B_2 = \bmat{0 \\
        1 \\
        1 \\
        0 }\, 
    B_3 = \bmat{0 \\
        0 \\
        1 \\
        1 }
        \\
    Q_1 =\bmat{0.48 & 0 & 0 & 0 \\
        0 & 0.64 & 0 & 0 \\
        0 & 0 & 0.74 & 0 \\
        0 & 0 & 0 & 0.71 },\ 
    Q_2 = \bmat{0.01 & 0 & 0 & 0 \\
        0 & 0.41 & 0 & 0 \\
        0 & 0 & 0.71 & 0 \\
        0 & 0 & 0 & 0.44 },\ 
    Q_3 = \bmat{1.00 & 0 & 0 & 0 \\
        0 & 0.55 & 0 & 0 \\
        0 & 0 & 0.86 & 0 \\
        0 & 0 & 0 & 0.63 }, 
        \\
    R_{11} = \bmat{5.47 },\ 
    R_{12} = \bmat{7.16 },\ 
    R_{13} = \bmat{5.31 },\ 
    R_{21} = \bmat{5.21 },\ 
    R_{22} = \bmat{5.36 },\\
    R_{23} = \bmat{7.63 },\ 
    R_{31} = \bmat{9.71 },\ 
    R_{32} = \bmat{2.34 },\ 
    R_{33} = \bmat{5.26 },\ 
\end{gather*}
and \begin{gather*}
    K_{1} = \bmat{0.023 \\
    -0.201 \\
    -0.228 \\
    0.104 }^T,
    K_{2} = \bmat{0.060 \\
    0.029 \\
    0.026 \\
    0.274 }^T,
    K_{3} = \bmat{0.033 \\
    0.082 \\
    0.138 \\
    0.177 }^T.  
\end{gather*}

We use the  following values for constants used in the LQ game: $\alpha=19.4$
and $2.90 \times 10^{5}$; hence, we use $\gamma = 1.52 \times 10^{-5}$.

\subsection{Coupled Riccati equations}
\label{app:coupled_riccati}

We require the following standard assumption adopted in LQ games.
\begin{assumption}
Either $(A,B_1,\sqrt{Q_1})$ or $(A,B_2,\sqrt{Q_2})$ is stabilizable-detectable.
\end{assumption}

Without loss of generality, we assume $(A,B_1,\sqrt{Q_1})$ is stabilizable-detectable.
We employ the following iterative Lyapunov  algorithm for finding the Nash
equilibrium to the linear quadratic game~\cite{gajic1995}:
\begin{description}[itemsep=0pt, topsep=5pt, leftmargin=25pt]

    \item[step 1.]  
    Initialize $P_1^{(0)}$ to be the unique positive definite solution to the
 Riccati equation,
    \begin{align}
   P_1^{(0)} =  A^TP_1^{(0)}A  -
   A^TP_1^{(0)}B_1\big(R_1+B_1^TP_1^{(0)}B_1\big)^{-1}B_1^TP_1^{(0)}A +Q_1,
    \end{align}
    and compute the corresponding gain matrix for player 1 by  
    \begin{align}
    K_1^{(0)} & = (R_1 + B_1^TP_1^{(0)}B_1)^{-1} B_1^TP_1^{(0)}A.
    \end{align}
   Solve for $P_2^{(0)}$ by
    \begin{align}
     P_2^{(0)} = \bar{A}^TP_2^{(0)}\bar{A}  - \bar{A}^TP_2^{(0)}B_2\big(R_2+B_2^TP_2^{(0)}B_2\big)^{-1}B_2^TP_2^{(0)}\bar{A} +Q_1
    \end{align}
    where $\bar{A} = A-B_1K_1^{(0)}$
    and compute the corresponding gain matrix for player 2 by
    \begin{align}
    K_2^{(0)} & = (R_2 + B_2^TP_2^{(0)}B_2)^{-1} B_2^TP_2^{(0)}\big(A-B_1K_1^{(0)}\big).
    \end{align}
    We note that initializing using this method ensures that the initial closed loop matrix $A-B_1K_1^{(0)}-B_2K_2^{(0)}$ is stable.  
\item[step 2.]  Given $P_1^{(k)}$, $P_2^{(k)}$, $K_1^{(k)}$, and $K_2^{(k)}$,
    update the feedback gains using the following update rules:
    \begin{align}
    K_1^{(k+1)} & = (R_{11} + B_1^TP_1^{(k)}B_1)^{-1} B_1^TP_1^{(k)}(A-B_2K_2^{(k)}) \\
    K_2^{(k+1)} & = (R_{22} + B_2^TP_2^{(k)}B_2)^{-1}B_2^TP_2^{(k)}(A-B_1K_1^{(k)})
    \end{align}
\item[step 3.]  Update the cost-to-go matrices by solving the
    Lyapunov equations:
    \begin{eqnarray*}
        P_1^{(k)} &=&
        (\bar A -B_2K_2^{(k)})^TP_1^{(k+1)}(\bar A -B_2 K_2^{(k)})
        + (K_1^{(k)})^TR_{11}K_1^{(k)} + (K_2^{(k)})^TR_{12}K_2^{(k)} + Q_1\\
        P_2^{(k)} &=&
        (\bar A -B_2K_2^{(k)})^TP_2^{(k+1)}(\bar A -B_2K_2^{(k)}) + (K_1^{(k)})^TR_{21}K_1^{(k)}  + (K_2^{(k)})^TR_{22}K_2^{(k)} + Q_2
    \end{eqnarray*}
\item[step 4.] Repeat \textbf{steps 2--3} until the gains converge.
\end{description}

The extension to $n$-players is fairly straightforward; more detail can be found in the
seminal reference~\cite{basar:1998aa}.

\bibliographystyle{plainnat}
\bibliography{UAI_2019_refs}

\end{document}